\documentclass[reqno]{amsart}
\usepackage[utf8x]{inputenc}  

\usepackage{setspace}
\usepackage[left=3.2cm, right=3.2cm, bottom=4cm]{geometry}                 

\usepackage{graphicx} 
\usepackage{pgf,tikz}
\usetikzlibrary{arrows}
\usepackage{amssymb}
\usepackage{pdfsync}
\usepackage{mathrsfs}
\usepackage{hyperref} 
\usepackage{verbatim} 
\usepackage{epstopdf}
\usepackage{bbm} 
\usepackage[colorinlistoftodos,prependcaption,textsize=tiny]{todonotes}
\usepackage{xargs}

\def\R{\mathbb{R}}

\def\N{\mathbb{N}}

\def\C{\mathbb{C}}

\def\SS{\mathbb{S}}

\newcommandx{\emanuel}[2][1=]{\todo[linecolor=green,backgroundcolor=green!25,bordercolor=black,#1]{#2}}

\newcommandx{\diogo}[2][1=]{\todo[linecolor=orange,backgroundcolor=orange!25,bordercolor=orange,#1]{#2}}

\newcommandx{\mateus}[2][1=]{\todo[linecolor=blue,backgroundcolor=blue!25,bordercolor=blue,#1]{#2}}

\newcommandx{\danger}[2][1=]{\todo[linecolor=red,backgroundcolor=red!25,bordercolor=blue,#1]{#2}}

\renewcommand{\d}{\text{\rm d}}

\newtheorem{theorem}{Theorem}

\newtheorem{proposition}[theorem]{Proposition}
\newtheorem{lemma}[theorem]{Lemma}

\makeatletter
\DeclareFontFamily{U}{tipa}{}
\DeclareFontShape{U}{tipa}{m}{n}{<->tipa10}{}
\newcommand{\arc@char}{{\usefont{U}{tipa}{m}{n}\symbol{62}}}%

\newcommand{\arc}[1]{\mathpalette\arc@arc{#1}}

\newcommand{\arc@arc}[2]{%
  \sbox0{$\m@th#1#2$}%
  \vbox{
    \hbox{\resizebox{\wd0}{\height}{\arc@char}}
    \nointerlineskip
    \box0
  }%
}
\makeatother

\numberwithin{equation}{section}

\allowdisplaybreaks

\newcommand{\intav}[1]{\mathchoice {\mathop{\vrule width 6pt height 3 pt depth  -2.5pt
\kern -8pt \intop}\nolimits_{\kern -6pt#1}} {\mathop{\vrule width
5pt height 3  pt depth -2.6pt \kern -6pt \intop}\nolimits_{#1}}
{\mathop{\vrule width 5pt height 3 pt depth -2.6pt \kern -6pt
\intop}\nolimits_{#1}} {\mathop{\vrule width 5pt height 3 pt depth
-2.6pt \kern -6pt \intop}\nolimits_{#1}}}

\newcommand{\intavl}[1]{\mathchoice {\mathop{\vrule width 6pt height 3 pt depth  -2.5pt
\kern -8pt \intop}\limits_{\kern -6pt#1}} {\mathop{\vrule width 5pt
height 3  pt depth -2.6pt \kern -6pt \intop}\nolimits_{#1}}
{\mathop{\vrule width 5pt height 3 pt depth -2.6pt \kern -6pt
\intop}\nolimits_{#1}} {\mathop{\vrule width 5pt height 3 pt depth
-2.6pt \kern -6pt \intop}\nolimits_{#1}}}

\title[Gradient bounds for radial maximal functions]{Gradient bounds for radial maximal functions}

\author[Carneiro]{Emanuel Carneiro}
\author[Gonz\'{a}lez-Riquelme]{Cristian Gonz\'{a}lez-Riquelme}

\address{
ICTP - The Abdus Salam International Centre for Theoretical Physics, Strada Costiera, 11, I - 34151, Trieste, Italy}

\address{IMPA - Instituto de Matem\'{a}tica Pura e Aplicada, Estrada Dona Castorina, 110, Jardim Bot\^{a}nico, Rio de Janeiro - RJ, Brazil, 22460-320.}

\email{carneiro@ictp.it}
\email{carneiro@impa.br}

\address{IMPA - Instituto de Matem\'{a}tica Pura e Aplicada, Estrada Dona Castorina, 110, Jardim Bot\^{a}nico, Rio de Janeiro - RJ, Brazil, 22460-320.}
\email{cristian@impa.br}

\date{\today}                                           
\begin{document}

\subjclass[2010]{42B25, 46E35, 31B05, 35J05, 35K08}
\keywords{Maximal operators, Sobolev spaces, bounded variation, convolution, sphere}
\begin{abstract} 
In this paper we study the regularity properties of certain maximal operators of convolution type at the endpoint $p=1$, when acting on radial data. In particular, for the heat flow maximal operator and the Poisson maximal operator, when the initial datum $u_0 \in W^{1,1}(
\R^d)$ is a radial function, we show that the associated maximal function $u^*$ is weakly differentiable and  
$$\|\nabla u^*\|_{L^1(\R^d)} \lesssim_d \|\nabla u_0\|_{L^1(\R^d)}.$$
This establishes the analogue of a recent result of H. Luiro for the uncentered Hardy-Littlewood maximal operator, now in a centered setting with smooth kernels. In a second part of the paper, we establish similar gradient bounds for maximal operators on the sphere $\mathbb{S}^d$, when acting on polar functions. Our study includes the uncentered Hardy-Littlewood maximal operator, the heat flow maximal operator and the Poisson maximal operator on $\mathbb{S}^d$.
\end{abstract}

\maketitle 

\section{Introduction}

\subsection{A brief historical perspective} Maximal operators are central objects of study in harmonic analysis. One of the most basic examples is  the centered Hardy-Littlewood maximal operator, denoted here by $M$. For $f \in L^1_{{\rm loc}}(\R^d)$ it is defined as
$$Mf(x) = \sup_{r >0} \frac{1}{m(B_r(x))}\int_{B_r(x)} |f(y)|\,\d y = \sup_{r >0}\  \intav{B_r(x)} |f(y)|\,\d y\,,$$
where $B_r(x) \subset \R^d$ is the open ball centered at $x$ with radius $r$, and $m(B_r(x))$ denotes its $d$-dimensional Lebesgue measure. The uncentered Hardy-Littlewood maximal operator, denoted here by $\widetilde{M}$, is defined analogously, taking the supremum over open balls that simply contain the point $x$ but that are not necessarily centered at $x$. The fundamental theorem of Hardy, Littlewood and Wiener states that $M: L^1(\R^d) \to L^{1,\infty}(\R^d)$ and $M: L^p(\R^d) \to L^{p}(\R^d)$, for $1 < p \leq \infty$, are bounded operators. The same holds for $\widetilde{M}$.

\smallskip

In the seminal paper \cite{Ki}, Kinnunen studied the action of the Hardy-Littlewood maximal operator on Sobolev functions, giving an elegant proof that $M: W^{1,p}(\R^d) \to W^{1,p}(\R^d)$ is bounded for $1 < p \leq \infty$. This work paved the way for several interesting contributions to the regularity theory of maximal operators over the past two decades, with interesting connections to potential theory and partial differential equations, see for instance \cite{ACP, BM, BRS, CH, CMa, CMP, HM, HLX, KL, KiSa, LW, Lu1, M, M2, ML, PPSS, Ra, Saa}.
One of the longstanding problems in this field is concerned with the regularity at the endpoint $p=1$. This is the $W^{1,1}$-problem, formally posed by Haj\l asz and Onninen in \cite{HO}: if $f \in W^{1,1}(\R^d)$, do we have that $Mf$ is weakly differentiable and 
$$\|\nabla M f\|_{L^1(\R^d)} \lesssim_d \, \|\nabla f\|_{L^1(\R^d)} \ ?$$

\smallskip

This problem has been settled affirmatively in dimension $d=1$, in the uncentered case by Tanaka \cite{Ta} and Aldaz and P\'{e}rez L\'{a}zaro \cite{AP}, and in the centered case by Kurka \cite{Ku}. The higher dimensional version is generally open, having been settled affirmatively only for the {\it uncentered} Hardy-Littlewood maximal operator $\widetilde{M}$ in the important case of {\it radial datum $f$}, by Luiro in \cite{Lu2}. This beautiful work of Luiro \cite{Lu2} is fundamental for the present paper, for we aim to extend it to different contexts.

\subsection{Maximal operators of convolution type on $\R^d$} We start by investigating the higher dimensional $W^{1,1}$-problem for certain centered maximal operators of convolution type associated to partial differential equations, {\it in the case of radial data}, establishing a result analogous to that of Luiro \cite{Lu2}. To our knowledge, this is the first instance of an affirmative result for centered maximal operators, in what concerns the boundedness of the variation, in the higher dimensional setting. 

\smallskip

We borrow the basic setup from \cite{CFS}. Let $\varphi: \R^d \times (0,\infty) \to \R$ be a nonnegative function such that \begin{equation*}
\int_{\R^d} \varphi(x, t)\,\d x = 1
\end{equation*} 
for each $t >0$. Assume also that, when $t \to 0$, the family $\varphi(\cdot, t)$ is an approximation of the identity, in the sense that $\lim_{t \to 0} \varphi(\cdot, t) * u_0(x) = u_0(x)$ for a.e. $x \in \R^d$, if $u_0 \in L^p(\R^d)$ for some $1 \leq p \leq \infty$. For an initial datum $u_0:\R^d \to \R$ we consider the evolution
\begin{equation}\label{Def_u_x_t}
u(x,t) = \big(|u_0|* \varphi(\cdot, t)  \big)(x)
\end{equation}
and the associated maximal function
\begin{equation}\label{Notation_u*}
u^*(x) := \sup_{t>0} u(x,t).
\end{equation}
Notice the use of the shorter notation $u^*$ for simplicity. One could also refer to \eqref{Notation_u*} as $M_{\varphi} u_0$. In this setting, note that the centered Hardy-Littlewood maximal operator corresponds to the kernel $\varphi(x,t) = \frac{1}{t^d m(B_1)}\chi_{B_1}(x/t)$. We consider here kernels $\varphi_{a,b}$ that are fundamental solutions of
\begin{equation*}
au_{tt} - bu_t + \Delta u = 0 \ \ \ \ {\rm in}  \ \ \  \R^d \times (0,\infty)\,,
\end{equation*}
with $a,b \geq 0$ and $(a,b) \neq (0,0)$. That is, the function $u(x,t)$ defined in \eqref{Def_u_x_t} solves this equation in the upper half-space with initial datum $u(x,0) = |u_0(x)|$. By appropriate space-time dilations it suffices to consider the following three nonnegative and radial decreasing kernels as basic profiles:
\begin{align}\label{Poisson_kernel}
\varphi_{1,0}(x,t) &= \frac{\Gamma \left(\frac{d+1}{2}\right)}{\pi^{(d+1)/2}}\ \frac{t}{(|x|^2 + t^2)^{(d+1)/2}} \ \ \ \ \ \ \ {\rm (Poisson \ kernel)}\\
\varphi_{0,1}(x,t) & = \frac{1}{(4 \pi t)^{d/2}}\ e^{-|x|^2/4t}  \ \ \ \ \ \ \ \ \ \ \ \ \ \ \ \ \ \ \ {\rm (Heat \ kernel)} \label{heat_kernel}\\
\varphi_{1,1}(x,t)  &= \int_{\R^d} e^{-t \big(\frac{-1 + \sqrt{1 + 16 \pi^2 |\xi|^2}}{2}\big)} \, e^{2\pi i x \cdot \xi}\,\d \xi.\label{Ker3}
\end{align}
The fact that \eqref{Ker3} is nonnegative and radial decreasing was proved in \cite{CFS}. The {\it Poisson maximal operator} and the {\it heat flow maximal operator}, given by the kernels \eqref{Poisson_kernel} and \eqref{heat_kernel} respectively, are the classical and most important examples we want to keep in mind, but our methods could be adapted to treat other maximal operators associated to differential equations. Our first result is the following.
\begin{theorem}\label{Thm1}
Let $\varphi$ be given by \eqref{Poisson_kernel}, \eqref{heat_kernel} or \eqref{Ker3}. If $u_0 \in W^{1,1}(\R^d)$ is radial, then $u^*$ is weakly differentiable and 
$$\|\nabla u^*\|_{L^1(\R^d)} \lesssim_d  \|\nabla u_0\|_{L^1(\R^d)}.$$
\end{theorem}
The intuitive idea behind the proof of this result is as follows. First we reduce matters to the study of nonnegative functions $u_0$ with some degree of smoothness, say Lipschitz. We are then able to invoke one of the main results of \cite{CFS, CS}, that in the {\it detachment set} $\{u^*> |u_0|\}$ the function $u^*$ is {\it subharmonic}. The proof of this fact relies on some of the qualitative properties of the underlying partial differential equations (e.g. maximum principles and semigroup property). As observed in \cite[Theorem 1 (iv)]{CFS}, this subharmonicity implies a control on the $L^2$-norm of $\nabla u^*$ by the $L^2$-norm of $\nabla u_0$. To arrive at the $L^1$-control we use the fact that $u^*$ is pointwise smaller than $\widetilde{M}u_0$. Hence, in the case of radial functions, we have a relatively well-behaved (i.e. subharmonic in the detachment set) function, namely $u^*$, that is trapped between $u_0$ and $\widetilde{M}u_0$, and the latter comes with an $L^1$-control of the gradient by the result of Luiro \cite{Lu2}. As we shall see, these pieces together will ultimately imply the control of the $L^1$-norm of $\nabla u^*$ as well.

\subsection{The Hardy-Littlewood maximal operator on $\mathbb{S}^{d}$} We now move our discussion to consider maximal operators acting on functions defined on the sphere $\mathbb{S}^{d} \subset \R^{d+1}$, in order to develop an analogous theory. First, let us establish the basic notation to be used in this context. We let $d(\zeta,\eta)$ denote the geodesic distance between two points $\zeta,\eta \in \mathbb{S}^{d}$. Let $\mathcal{B}_r(\zeta) \subset \mathbb{S}^{d}$ be the open geodesic ball of center $\zeta \in \mathbb{S}^{d}$ and radius $r >0$, that is
$$\mathcal{B}_r(\zeta) = \{ \eta \in \mathbb{S}^{d} \ : \ d(\zeta,\eta) < r\},$$
and let $\overline{\mathcal{B}_r(\zeta)}$ be the corresponding closed ball. Let $\widetilde{\mathcal{M}}$ denote the uncentered Hardy-Littlewood maximal operator on the sphere $\mathbb{S}^{d}$, that is, for $f \in L^1_{{\rm loc}}(\mathbb{S}^{d})$,
$$\widetilde{\mathcal{M}}f(\xi) = \sup_{\{\overline{\mathcal{B}_r(\zeta)} \ : \ \xi \in \overline{\mathcal{B}_r(\zeta)}\}} \frac{1}{\sigma(\mathcal{B}_r(\zeta))}\int_{\mathcal{B}_r(\zeta)} |f(\eta)|\,\d \sigma(\eta) = \sup_{\{\overline{\mathcal{B}_r(\zeta)} \ : \ \xi \in \overline{\mathcal{B}_r(\zeta)}\}}\  \intav{\mathcal{B}_r(\zeta)} |f(\eta)|\,\d \sigma(\eta),$$
where $\sigma = \sigma_d$ denotes the usual surface measure on the sphere $\mathbb{S}^{d}$. The centered version ${\mathcal{M}}$ would be defined with centered geodesic balls. Fix ${\bf e} = (1, 0,0,\ldots,0) \in \R^{d+1}$ to be our north pole. We say that a function $f: \mathbb{S}^{d} \to \C$ is {\it polar} if for every $\xi, \eta \in \mathbb{S}^{d}$ with $ \xi \cdot {\bf e}  =  \eta \cdot {\bf e} $ we have $f(\xi) = f(\eta)$. This will be the analogue, in the spherical setting, of a radial function in the Euclidean setting.

\smallskip

When working on the circle $\mathbb{S}^1$, an adaptation of the proof of Aldaz and P\'{e}rez L\'{a}zaro \cite{AP} yields ${\rm Var} (\widetilde{\mathcal{M}}f) \leq {\rm Var} (f)$, where ${\rm Var} (f)$ denotes the total variation of the function $f$. This follows from the fact that $\widetilde{\mathcal{M}}f$ has no local maxima in the detachment set $\{\widetilde{\mathcal{M}}f >|f|\}$ (say, for $f$ Lipschitz). Our second result is the extension of this statement to the multidimensional setting, in the case of polar functions. For the basic theory of Sobolev spaces on the sphere $\mathbb{S}^{d}$ we refer the reader to \cite{Dai_Xu}.
    
\begin{theorem}\label{Thm2-sphere}
If $f \in W^{1,1} (\mathbb{S}^{d})$ is a polar function, then $\widetilde{\mathcal{M}}f$ is weakly differentiable and 
$$\|\nabla \widetilde{\mathcal{M}}f\|_{L^1(\mathbb{S}^{d})} \lesssim_d \|\nabla f\|_{L^1(\mathbb{S}^{d})}.$$
\end{theorem}

This is the analogue on the sphere $\mathbb{S}^{d}$ of Luiro's result  \cite{Lu2} for radial functions in the Euclidean space. The proof we present below follows broadly the strategy outlined by Luiro \cite{Lu2}. However, due to the different geometry, several nontrivial technical points arise along the proof and must be considered carefully. A good example that such difficulties cannot be underestimated is Lemma \ref{crucial_lemma_large_radii} below, one of the core results used in our proof of Theorem \ref{Thm2-sphere}. As in the case of $\R^d$, the analogue of Theorem \ref{Thm2-sphere} for the centered Hardy-Littlewood maximal operator ${\mathcal{M}}$ on $\mathbb{S}^{d}$ is an open problem.

\subsection{Maximal operators of convolution type on $\mathbb{S}^{d}$} We now treat two important cases of maximal operators of convolution type on the sphere: the Poisson maximal operator and the heat flow maximal operator.  We briefly recall the basic definitions and refer the reader to \cite[Section 1.4]{CFS} for additional details.

\subsubsection{Poisson maximal function on $\mathbb{S}^{d}$} Let $0\leq \rho < 1$ and let $\xi, \eta \in \mathbb{S}^{d}$. We define the Poisson kernel $\mathcal{P}$ on the sphere by 
\begin{equation*}\label{Poisson_kernel_sphere}
\mathcal{P}(\xi,\eta,\rho)=\frac{1-\rho^2}{ \kappa_{d}\,|\rho \xi-\eta|^d} = \frac{1-\rho^2}{ \kappa_{d} \,(\rho^2 - 2\rho\, \xi\cdot \eta + 1)^{d/2}}\,,
\end{equation*} 
with $\kappa_{d} = \sigma(\mathbb{S}^{d})$ being the total surface area of $\mathbb{S}^{d}$. If $u_0\in L^1(\mathbb{S}^{d})$ we let $u(\xi,\rho)=u(\rho\xi)$ be the function defined on the unit $(d+1)$-dimensional open ball $B_1 \subset \R^{d+1}$ by
\begin{equation*}
u(\xi,\rho)= \int_{\mathbb{S}^{d}}\mathcal{P}(\xi,\eta, \rho)\,|u_0(\eta)|\,\d\sigma(\eta)\,,
\end{equation*}
and consider the associated maximal function
\begin{equation}\label{Max_Poisson_sphere}
u^*(\xi)=\sup_{0\leq\rho<1} u(\xi,\rho).
\end{equation}
Observe that $u \in C^{\infty}(B_1)$ and solves the Dirichlet problem
\begin{equation*}
\left\{\begin{array}{ll}
 \Delta u = 0 & {\rm in} \ B_1\,; \\
\displaystyle\lim_{\rho \rightarrow 1^-}{u(\xi,\rho)}=|u_0(\xi)| & \mathrm{for~a.e.}~\xi\in \mathbb{S}^{d}.
\end{array}\right.
\end{equation*}

\subsubsection{Heat flow maximal function on $\mathbb{S}^{d}$} Let $\big\{Y_n^{\ell}\big\}$, $\ell = 1,2,\ldots, {\rm dim}\,{H}_n^{d+1}$, be an orthonormal basis of the space ${H}_n^{d+1}$ of spherical harmonics of degree $n$ in the sphere $\mathbb{S}^d$. For $t \in (0,\infty)$ and $\xi, \eta \in \mathbb{S}^{d}$ we define the heat kernel $\mathcal{K}$ on the sphere (see \cite[Lemma 1.2.3, Theorem 1.2.6 and Eq. 7.5.5]{Dai_Xu}) by
\begin{align*}
\mathcal{K}(\xi,\eta,t) & = \sum_{n=0}^{\infty} e^{-t n (n+d -1)} \sum_{\ell=1}^{{\rm dim}\,{H}_n^{d+1}} Y_n^{\ell}(\xi)Y_n^{\ell}(\eta) = \sum_{n=0}^{\infty} e^{-t n (n+d -1)}\frac{(n +\lambda)}{\lambda} \,C_n^{\lambda}(\xi \cdot \eta),
\end{align*}
where $\lambda = \frac{d-1}{2}$ and $t \mapsto C^{\beta}_n(t)$, for $\beta >0$, are the {\it Gegenbauer polynomials} defined in terms of the generating function 
\begin{equation*}
(1 - 2rt + r^2)^{-\beta} = \sum_{n=0}^{\infty} C^{\beta}_n(t)\, r^n.
\end{equation*}
If $u_0\in L^1(\mathbb{S}^{d})$ we consider
\begin{equation*}
u(\xi,t)= \int_{\mathbb{S}^{d}}\mathcal{K}(\xi,\eta, t)\,|u_0(\eta)|\,\d\sigma(\eta)\,,
\end{equation*}
and consider the associated maximal function
\begin{equation}\label{Max_heat_sphere}
u^*(\xi)=\sup_{t >0} u(\xi,t).
\end{equation}
Note that $u$ is a smooth function on $\mathbb{S}^{d} \times (0,\infty)$ and solves the heat equation
\begin{equation*}
\left\{\begin{array}{ll}
\partial_tu - \Delta u = 0 & {\rm in} \ \mathbb{S}^d \times (0,\infty)\,; \\
 \displaystyle\lim_{t \rightarrow 0^+}{u(\xi,t)}=|u_0(\xi)| & \mathrm{for~a.e.}~\xi\in \mathbb{S}^{d}.
\end{array}\right.
\end{equation*}
\subsubsection{Gradient bounds} We note that the smooth kernels $\mathcal{P}$ and $\mathcal{K}$ depend only on $d(\xi, \eta)$ and are decreasing with respect to this distance. If we fix one of these two parameters, they have integral $1$ on $\mathbb{S}^{d}$ and are approximate identities as $\rho \to 1^-$ and $t \to 0^+$, respectively. The discussion on the heat kernel can be found in \cite[Chapter III, Section 2]{SY}. Also, from \cite[Chapter 2, Theorem 2.3.6]{Dai_Xu}, note that the associated maximal functions $u^*$ are dominated by the Hardy-Littlewood maximal function, that is
\begin{equation}\label{20200625_11:06}
u^*(\xi) \leq  \mathcal{M}u_0(\xi) \leq  \widetilde{\mathcal{M}}u_0(\xi).
\end{equation}
Our third result establishes the following.
\begin{theorem}\label{Thm_conv_spheres}
Let $u^*$ be the Poisson maximal function given by \eqref{Max_Poisson_sphere} or the heat flow maximal function given by \eqref{Max_heat_sphere}. If $u_0 \in W^{1,1}(\mathbb{S}^d)$ is a polar function, then $u^*$ is weakly differentiable and 
$$\|\nabla u^*\|_{L^1(\mathbb{S}^d)} \lesssim_d \|\nabla u_0\|_{L^1(\mathbb{S}^d)}.$$
\end{theorem}

\subsection{A word on notation} In what follows we write $A \lesssim_d B$ if $A \leq C B$ for a certain constant $C>0 $ that may depend on the dimension $d$. We say that $A  \simeq_d B$ if $A \lesssim_d B$ and $B \lesssim_d A$. If there are other parameters of dependence, they will also be indicated. The characteristic function of a generic set $H$ is denoted by $\chi_H$. In the few occasions that we write universal constants $C_d$ in Section \ref{Section 4}, these may change from line to line.

\section{Proof of Theorem \ref{Thm1}}

In this section we prove Theorem \ref{Thm1}. Without loss of generality we may assume that $u_0$ is {\it real-valued and nonnegative} (or $+\infty$). Assume also that $d \geq 2$, since the result is already known for dimension $d=1$ from \cite[Theorem 1]{CFS}. Throughout the proof below, with a slight abuse of notation, we identify radial functions of the variable $x \in \R^d$ with their one-dimensional versions of the variable $r \in (0,\infty)$, with the understanding that $r = |x|$. Naturally, if $u_0$ is radial, the maximal function $u^*$ is also radial. In what follows, variables $r,s,t,\tau, a,b$ will be one-dimensional, whereas the variable $x$  is always reserved for $\R^d$. We recall the fact \cite[Chapter III, Theorem 2]{S} that
\begin{equation}\label{Stein_bound}
u^*(x)   \leq Mu_0(x) \leq \widetilde{M}u_0(x)
\end{equation}
for every $x \in \R^d$. 

\subsection{Lipschitz case} Let us first assume that our initial datum $u_0$ is a Lipschitz function. In this case $u^*$ is also Lipschitz. Reducing matters to radial variables, we claim the following:
\begin{equation}\label{Feb25_12:42}
\int_0^{\infty} \big|(u^*)'(r)\big|\, r^{d-1}\, \d r \leq \int_0^{\infty} \big|u_0'(r)\big|\, r^{d-1}\, \d r + \int_0^{\infty} \big| \big(\widetilde{M}u_0\big)'(r)\big|\, r^{d-1}\, \d r.
\end{equation}
Once we have established \eqref{Feb25_12:42}, the theorem follows easily by Luiro's result \cite{Lu2}, that bounds the third integral in terms of the second. 

\subsubsection*{Step 1: Partial control by the uncentered maximal function}Let us define the radial detachment set (excluding the origin)
\begin{equation}\label{Def_A}
A_d = \big\{ x \in \R^d \setminus \{0\} \ : \  u^*(x) > u_0(x)\big\}.
\end{equation}
The one-dimensional radial version of this set will be denoted by
\begin{equation*}
A_1 = \{ |x|  \ : \ x \in A_d \}.
\end{equation*} 
These are open sets and from \cite[Lemma 7]{CFS} we know that $u^*$ is subharmonic on $A_d$. Let us write 
\begin{equation}\label{disjoint union}
A_1 = \bigcup_{i=1}^{\infty} (a_i, b_i)
\end{equation}
as a countable union of disjoint open intervals. Let $(a,b)$ denote a generic interval $(a_i,b_i)$ of this union. If $u^*$ had a strict local maximum in $(a,b)$ (that is, a point $t_0 \in (a,b)$ for which there exist $c$ and $d$ with $a < c < t_0 < d < b$ such that $u^*(r) \leq u^*(t_0)$ for $r \in (c,d)$ and $u^*(c), u^*(d) < u^*(t_0)$), we could then take the average of $u^*$ over the ball in $\R^d$ centered at $x_0$, with $|x_0| = t_0$, and radius $\min\{|t_0 -c|, |t_0 -d|\}$ to reach a contradiction to the subharmonicity of $u^*$ in $A_d$. Therefore $u^*$ has no strict local maximum in $(a,b)$ and there exists $\tau$ with $a \leq \tau \leq b$ such that $u^*$ is non-increasing in $[a, \tau]$ and non-decreasing in $[\tau, b]$. We then have $(u^*)'(t) \leq 0$ a.e. in $a < t < \tau$, and $(u^*)'(t) \geq 0$ a.e. in $\tau < t < b$. 

\smallskip

Let us first consider the case $0 < a < b < \infty$. Using \eqref{Stein_bound} and integration by parts we get
\begin{align}
\int_a^{b} \big|(u^*)'(r)\big|\,&  r^{d-1}\, \d r = - \int_a^{\tau} (u^*)'(r)\, r^{d-1}\, \d r + \int_\tau^{b} (u^*)'(r)\, r^{d-1}\, \d r \nonumber\\
& = u^*(a)\, a^{d-1} + u^*(b)\, b^{d-1} - 2\,u^*(\tau)\, \tau^{d-1} \nonumber\\
&  \ \ \ \ \ \ \ + (d-1)\int_a^{\tau} u^*(r)\, r^{d-2}\, \d r - (d-1)\int_\tau^{b} u^*(r)\, r^{d-2}\, \d r \nonumber\\
& \leq  u_0(a)\, a^{d-1} + u_0(b)\, b^{d-1} - 2\,u_0(\tau)\, \tau^{d-1} \nonumber\\
&  \ \ \ \ \ \ \ + (d-1) \int_a^{\tau} \widetilde{M}u_0(r)\, r^{d-2}\, \d r - (d-1) \int_\tau^{b} u_0(r)\, r^{d-2}\, \d r \nonumber\\
& =  u_0(a)\, a^{d-1}  - \,u_0(\tau)\, \tau^{d-1}\nonumber\\
&  \ \ \ \ \ \ \ + (d-1) \int_a^{\tau} \widetilde{M}u_0(r)\, r^{d-2}\, \d r + \int_\tau^{b} u_0'(r)\, r^{d-1}\, \d r\nonumber\\
& \leq \int_a^{b} \big|u_0'(r)\big|\, r^{d-1}\, \d r + (d-1) \int_a^{\tau} \widetilde{M}u_0(r)\, r^{d-2}\, \d r. \label{Main_chain_pf}
\end{align}
The last inequality holds since 
\begin{equation*}
 u_0(a)\, a^{d-1}  - \,u_0(\tau)\, \tau^{d-1} \leq - \int_a^{\tau} u_0'(r)\, r^{d-1}\, \d r \leq  \int_a^{\tau} \big|u_0'(r)\big|\, r^{d-1}\, \d r.
\end{equation*}

\smallskip

If $b = \infty$, since $u^* \in L^{1,\infty}(\R^d)$ we must have $\tau = \infty$ as well (i.e. $u^*$ non-increasing in the interval $(a,\infty)$) and a simple limiting argument leads to inequality \eqref{Main_chain_pf} again. Note that $\lim_{r \to \infty} u_0(r)\,r^{d-1} = 0$ since $r \mapsto u_0(r)\,r^{d-1}$ is locally Lipschitz with integrable derivative in $(0,\infty)$. 

\smallskip 

Finally, if $a=0$, the proof of \eqref{Main_chain_pf} follows as above noting that $\lim_{r \to 0} u^*(r)\,r^{d-1} = 0$ (for $d \geq 2$).

\smallskip

If we add up \eqref{Main_chain_pf} over all the intervals $(a_i, b_i)$ of the disjoint union \eqref{disjoint union} we find
\begin{equation*}
\int_{A_1} \big|(u^*)'(r)\big|\,  r^{d-1}\, \d r \leq \int_{A_1} \big|u_0'(r)\big|\, r^{d-1}\, \d r + (d-1) \int_0^{\infty} \widetilde{M}u_0(r)\, r^{d-2}\, \d r\,,
\end{equation*}
which then leads to (note that in $A_{1}^c$ we have $u^* = u_0$, and hence $(u^*)' = u_0'$ a.e. in $A_{1}^c$).
\begin{equation}\label{Feb25_15:41}
\int_0^{\infty} \big|(u^*)'(r)\big|\,  r^{d-1}\, \d r \leq \int_0^{\infty} \big|u_0'(r)\big|\, r^{d-1}\, \d r + (d-1) \int_0^{\infty} \widetilde{M}u_0(r)\, r^{d-2}\, \d r.
\end{equation}

\subsubsection*{Step 2: Control of weighted norms} As $r \mapsto \widetilde{M}u_0(r)$ is Lipschitz and its derivative is integrable (in fact $\big(\widetilde{M}u_0\big)'\!(r)\, r^{d-1} \in L^1(0,\infty)$ from Luiro's work \cite{Lu2}) we have that $\lim_{r \to \infty} \widetilde{M}u_0(r)$ exists and it is equal to $0$ since $\widetilde{M}u_0 \in L^{1,\infty}(\R^d)$. Then 
$$ \widetilde{M}u_0(r) = - \int_r^{\infty} \big(\widetilde{M}u_0\big)'\!(t)\, \d t$$ 
and 
\begin{align}\label{Feb25_15:42}
\begin{split}
(d-1) \int_0^{\infty} \widetilde{M}u_0(r)\, r^{d-2}\, \d r & = (d-1) \int_0^{\infty}\left( \int_r^{\infty} - \big(\widetilde{M}u_0\big)'\!(t)\, \d t\right) r^{d-2}\, \d r  \\
& \leq (d-1) \int_0^{\infty} \left(\int_r^{\infty}  \big|\big(\widetilde{M}u_0\big)'\!(t)\big|\, \d t\right) r^{d-2}\, \d r  \\
& =  (d-1) \int_0^{\infty} \int_0^{t}  r^{d-2} \, \big|\big(\widetilde{M}u_0\big)'\!(t)\big|\, \d r \,\d t \\
& = \int_0^{\infty}  \big|\big(\widetilde{M}u_0\big)'\!(t)\big|\, t^{d-1} \,\d t.
\end{split}
\end{align}
Finally, we combine \eqref{Feb25_15:41} and \eqref{Feb25_15:42} to arrive at \eqref{Feb25_12:42}, concluding the proof in this case.

\subsection{General case} \label{General case_Convolution}
Let us first record a basic lemma about radial functions and weak derivatives. In what follows, when we say that a function $f$ is weakly differentiable in a certain domain $\Omega \subset \R^d$, it is naturally understood that $f$ and its weak derivatives are locally integrable in such a domain.

\begin{lemma}\label{Lem_radialization}.
\begin{itemize}
\item[(i)] A radial function $f(x)$ is weakly differentiable in $\R^d \setminus \{0\}$ if and only if its radial restriction $f(r)$ is weakly differentiable in $(0,\infty)$. In this case, the weak gradient $\nabla f$ of $f(x)$ and the weak derivative $f'$ of $f(r)$ are related by $\nabla f(x) = f'(|x|)\frac{x}{|x|}$. 

\smallskip

\item[(ii)] In the situation above, if $f(x)$ and $\nabla f(x)$ are locally integrable in a neighborhood of the origin, then $f$ is weakly differentiable in $\R^d$.
\end{itemize}
\end{lemma}
\begin{proof} This result is most certainly standard but we could not find an exact explicit reference. We then provide a brief proof for completeness.

\smallskip

\noindent{{\it Part} (i).} Assume that $f(x)$ is weakly differentiable in $\R^d \setminus \{0\}$ and let $\nabla f $ be its weak gradient. Let $\varphi \in C^{\infty}_c(\R^d \setminus \{0\})$ be a radial test function. Letting $r = |x|$ we have, by definition,
\begin{align}\label{20200629_09:40}
\begin{split}
\int_{\R^d \setminus \{0\}} f(x) \left( \frac{(d-1)}{|x|} \varphi(x) + \frac{\partial \varphi}{\partial r}(x)\right) \d x = \int_{\R^d \setminus \{0\}}  f(x) \left( \sum_{i=1}^{d}\frac{\partial}{\partial x_i} \left(\frac{x_i}{|x|} \varphi(x)\right) \right) \d x \\
=  - \int_{\R^d \setminus \{0\}} \left(\sum_{i=1}^{d}\frac{\partial f}{\partial x_i} \frac{x_i}{|x|} \varphi(x)\!\right)  \d x = - \int_{\R^d \setminus \{0\}}  \nabla f (x)\cdot \frac{x}{|x|} \,\,\varphi(x)\,\d x. 
\end{split}
\end{align}
Write $x = r\omega$, with $\omega \in \mathbb{S}^{d-1}$. Letting $\Phi(r) =  \varphi(r)\,r^{d-1}$, rewrite \eqref{20200629_09:40} in polar coordinates to get
\begin{equation*}
\sigma_{d-1}\big(\mathbb{S}^{d-1}\big) \int_{0}^{\infty} f(r)\,\Phi'(r)\,\d r = - \int_{0}^{\infty} \left( \int_{\mathbb{S}^{d-1}} (\nabla f (r\omega)) \cdot \omega \,\, \d \sigma_{d-1}(\omega)\right)\,\Phi(r) \,\d r.
\end{equation*}
This is the required integration by parts in $(0,\infty)$ for the generic test function $\Phi$.

\smallskip

Assume now that $f(r)$ is weakly differentiable in $(0,\infty)$. If $g$ is its weak derivative, then $f(r) - \int_{1}^{r} g(t)\,\d t$ has weak derivative zero and hence is constant a.e. in $(0,\infty)$. We can then modify $f$ on a set of measure zero so that $f$ is continuous in $(0,\infty)$; in fact absolutely continuous in each interval $[a,b] \subset (0,\infty)$. In particular, $f$ is differentiable a.e. and $g = f'$. The radial extension $f(x)$ is then continuous in $\R^d \setminus \{0\}$ and differentiable almost everywhere. Let us show that integration by parts holds, say, with respect to the first coordinate $x_1$. Write $x = (x_1, x_1, \ldots, x_d) = r \omega = (r\cos \theta, r (\sin \theta) \xi)$, with $r \in (0,\infty)$, $\omega \in \mathbb{S}^{d-1} \subset \R^d$, $0 \leq \theta \leq \pi$ and $\xi \in \mathbb{S}^{d-2} \subset \R^{d-1}$. Let $\psi \in C^{\infty}_c(\R^d \setminus \{0\})$ be a generic test function and consider
\begin{equation*}
\Psi(r) = \left(\int_{\mathbb{S}^{d-1}} \psi \,\, \frac{x_1}{|x|}  \,\, \d \sigma_{d-1}(\omega)\right)  r^{d-1} = \left(\int_0^{\pi} \left( \int_{\mathbb{S}^{d-2}} \psi  \,\, \d \sigma_{d-2}(\xi)\right) \cos \theta \, (\sin \theta)^{d-2} \,\d \theta \right) r^{d-1}.
\end{equation*}
Then 
\begin{align*}
\Psi'(r) &=  \left(\int_0^{\pi}  \int_{\mathbb{S}^{d-2}} \left(\frac{\partial \psi}{\partial r} \cos \theta  - \frac{\partial \psi}{\partial \theta} \frac{\sin \theta}{r} \right)  (\sin \theta)^{d-2} \,\d \sigma_{d-2}(\xi) \, \d \theta \right)\,r^{d-1}
\end{align*}
where an integration by parts in the variable $\theta$ was used. Using polar coordinates one now sees that 
\begin{align*}
\int_{\R^d \setminus \{0\}} f(x) \,\frac{\partial \psi}{\partial x_1}\, \d x = \int_0^{\infty} f(r) \, \Psi'(r) \,\d r  = - \int_0^{\infty} f'(r) \, \Psi(r) \,\d r = - \int_{\R^d \setminus \{0\}} \left(f'(|x|) \,\frac{x_1}{|x|}\right) \,\psi(x) \,\d x.
\end{align*}
This shows that $f(x)$ is weakly differentiable with weak gradient given by $\nabla f(x) = f'(|x|)\frac{x}{|x|}$.

\smallskip

\noindent{{\it Part} (ii).}  Let $\psi:\R^d \to \R$ be a smooth radial non-increasing function with $\psi \equiv 1$ on $\{|x| \leq 1\}$ and $\psi \equiv 0$ on $\{|x| \geq 2\}$. Let $\Psi_\alpha(x) = 1 -\psi(x/\alpha)$. Let $\phi \in C^{\infty}_c(\R^d)$ be any test function. Since we know that $f$ is weakly differentiable in $\R^d \setminus \{0\}$ we have, for any direction $i=1,2,\ldots,d$ (here we denote $\partial f / \partial x_i$ simply by $f_{x_i}$),
\begin{align}\label{Fev27_16:46}
\begin{split}
- \int_{\R^d}  f_{x_i}(x)\, (\phi \Psi_{\alpha})(x)\, \d x  &= \int_{\R^d} f(x)\,(\phi \Psi_{\alpha})_{x_i}(x)\, \d x\\
& = \int_{\R^d} f(x)\,\phi_{x_i} (x)\,\Psi_{\alpha}(x) \,\d x + \int_{\R^d} f(x)\,\phi(x)\,(\Psi_{\alpha})_{x_i} (x)\, \d x. 
\end{split}
\end{align} 
Note that the last integral takes place inside the ball of radius $2\alpha$. In this ball we have $\phi(x) = \phi(0) + R(x)$ with $|R(x)| \leq C\alpha$. Since $f(x)$ is even in the variable $x_i$ and $(\Psi_{\alpha})_{x_i} (x)$ is odd in the variable $x_i$ we get 
\begin{equation}\label{Fev27_16:47}
\int_{\R^d} f(x)(\Psi_{\alpha})_{x_i} (x)\, \d x = 0\,,
\end{equation}
and since $(\Psi_{\alpha})_{x_i} (x) = -\frac{1}{\alpha} \psi_{x_i} (x/\alpha)$ we find 
\begin{equation}\label{Fev27_16:48}
\int_{\R^d} f(x)R(x)(\Psi_{\alpha})_{x_i} (x)\, \d x \to  0
\end{equation}
as $\alpha \to 0$, since $f$ is locally integrable. Using \eqref{Fev27_16:47} and \eqref{Fev27_16:48} and the fact that $\nabla f$ is also locally integrable we may pass the limit as $\alpha \to 0$ in \eqref{Fev27_16:46} to find
$$ -\int_{\R^d}  f_{x_i}(x)\, \phi(x)\, \d x  = \int_{\R^d} f(x)\,\phi_{x_i} (x)\,\d x\,,$$
as desired.
\end{proof}

We now consider the case of general $u_0 \in W^{1,1}(\R^d)$ radial. We have seen in Lemma \ref{Lem_radialization} that its radial version $u_0(r)$ is weakly differentiable in $(0,\infty)$ and 
$$\int_0^{\infty} |u_0'(r)|\,r^{d-1}\,\d r <\infty.$$
In particular, after a possible redefinition on a set of measure zero, one can take $u_0(r)$ continuous in $(0,\infty)$ (in fact, absolutely continuous in each interval $[a,\infty)$ for $a>0$). This is equivalent to assuming that $u_0(x)$ is continuous in $\R^d\setminus\{0\}$.

\subsubsection*{Step 3: $u^*$ is continuous in $\R^d \setminus\{0\}$} With  $u_0(x)$ continuous in $\R^d\setminus\{0\}$, the detachment set $A_d$ defined in \eqref{Def_A} is open. Throughout the rest of this section let us write
$$u_\varepsilon (x):= u(x, \varepsilon)= \big(u_0*\varphi(\cdot,\varepsilon)\big)(x), \ \ x \in \R^d, \ \varepsilon >0.$$
We claim that $u^*$ is locally Lipschitz in $A_d$. In fact, if $x_0 \in A_d$, there exists $t_0 >0$ such that 
$$u^*(x_0) = u(x_0,t_0 )  > u(x_0).$$
From the continuity of $u(x,t)$, there exist a neighborhood $V$ of $x_0$ and an $\varepsilon_0 >0$ such that 
\begin{equation}\label{Feb27_15:21}
u^*(x) =\sup_{t> 0} u(x,t) =  \sup_{t> \varepsilon_0} u(x,t) = \sup_{t> 0} \big(u_{\varepsilon_0}*\varphi(\cdot,t)\big)(x) =: u_{\varepsilon_0}^*(x)
\end{equation}
for all $x \in V$. Note that in the third equality above we used the semigroup property of the family $\varphi(\cdot,t)$ (i.e. the fact that $\varphi(\cdot, t_1)*\varphi(\cdot, t_2) = \varphi(\cdot, t_1+t_2)$).  Since $u_{\varepsilon_0}$ is Lipschitz, we have that $u^* = u_{\varepsilon_0}^*$ is Lipschitz on $V$, which proves our claim. 

\smallskip

Writing $\R^d \setminus\{0\} = A_d \cup A_d^{c}$, we now need so show that $u^*$ is continuous at the points of $A_d^{c}$. Let $x_0 \in A_d^{c}$. If $x_0 \in {\rm int} (A_d^{c})$ we are done since $u^* = u_0$ is continuous in a neighborhood of $x_0$. Assume now that $x_0 \in A_d^{c} \setminus {\rm int} (A_d^{c})$ and that there exists a sequence $\{x_n\}_{n\in \N} \subset A_d$ such that $x_n \to x_0$ but $u^*(x_n) \nrightarrow u^*(x_0) = u_0(x_0)$. Then there exist $t_n>0$ and $\delta >0$ such that $u(x_n, t_n) \geq u_0(x_0) + \delta$ for all $n$. From the integrability of $u_0$, the $t_n$ are bounded, and passing to a subsequence we may assume that $t_n \to t \geq 0$. Then $u(x_n, t_n) \to u(x_0,t) \geq u_0(x_0) + \delta$, and we get that $t>0$ and $x_0 \in A_d$, a contradiction. This establishes that $u^*$ is continuous in $\R^d \setminus \{0\}$.

\subsubsection*{Step 4: Weak differentiability and conclusion} In the previous step we showed that $u^*(r)$ is continuous on $(0,\infty)$ and locally Lipschitz in $A_1$. For almost every $r \in A_1$, from \eqref{Feb27_15:21} we have
$$(u^*)'(r) = \lim_{\varepsilon \to 0} (u_{\varepsilon}^*)'(r).$$
From Minkowski's inequality we recall that 
\begin{equation}\label{Minkowski}
\|\nabla u_{\varepsilon}\|_{L^1(\R^d)} \leq \|\nabla u_0\|_{L^1(\R^d)}
\end{equation}
for any $\varepsilon >0$. Using Fatou's lemma, the bound in Theorem \ref{Thm1} already proved for Lipschitz functions, and \eqref{Minkowski}, we arrive at
\begin{align}\label{Feb27_16:53}
\begin{split}
\int_{A_1} \big|(u^*)'(r)\big|\,  r^{d-1}\, \d r &\leq \liminf_{\varepsilon \to 0}\int_{A_1} \big|(u_{\varepsilon}^*)'(r)\big|\,  r^{d-1}\, \d r \\
& \lesssim_d \liminf_{\varepsilon \to 0} \|\nabla u_{\varepsilon}\|_{L^1(\R^d)}\\
& \leq  \|\nabla u_0\|_{L^1(\R^d)}.
\end{split}
\end{align}
With this in hand, an adaptation of the argument in \cite[Section 5.4]{CS} shows that $u^*(r)$ is weakly differentiable in $(0,\infty)$ with weak derivative given by $\chi_{A_1^c} u_0'(r) + \chi_{A_1} (u^*)'(r)$. This in turn implies that $u^*(x)$ is weakly differentiable in $\R^d \setminus \{0\}$ by Lemma \ref{Lem_radialization}. From \eqref{Feb27_16:53}, its weak gradient $\nabla u^*$ on $\R^d \setminus \{0\}$ verifies
\begin{align}\label{Bound_grad_caso_gen}
\|\nabla u^*\|_{L^1(\R^d)} & = \kappa_{d-1}\int_0^{\infty} \big|(u^*)'(r)\big|\,  r^{d-1}\, \d r \nonumber \\
& = \kappa_{d-1} \left(\int_{A_1} \big|(u^*)'(r)\big|\,  r^{d-1}\, \d r + \int_{A_1^c} \big|u_0'(r)\big|\,  r^{d-1}\, \d r \right)\\
& \lesssim_d \|\nabla u_0\|_{L^1(\R^d)},  \nonumber
\end{align}
with $\kappa_{d-1}$ being the total surface measure of $\mathbb{S}^{d-1}$. This is our desired bound. As a final remark note that, from the Sobolev embedding, $u_0 \in L^{d/(d-1)}(\R^d)$ and hence so does $u^*$. In particular, $u^*$ is locally integrable in $\R^d$. Since we already know from \eqref{Bound_grad_caso_gen} that $\nabla u^* \in L^1(\R^d)$, an application of Lemma \ref{Lem_radialization} (ii) gives us that $u^*$ is in fact weakly differentiable in $\R^d$. This completes the proof of Theorem \ref{Thm1}.

\smallskip

\noindent{\it Remark}: A crucial insight in the proof above was to relate the variation of $u^*$ with the variation of the uncentered Hardy-Littlewood maximal operator $\widetilde{M}u_0$, expressed in inequality \eqref{Feb25_12:42}. Since $\widetilde{M}u_0(x) \lesssim_d  u^*(x) $, uniformly for all $x \in \R^d$, we could just run the exact same proof to obtain the gradient bound for $\widetilde{M}u_0$ starting from the gradient bound for $u^*$, showing that these two bounds are actually equivalent to each other.

\section{Proof of Theorem \ref{Thm2-sphere}}

Recall that $\sigma$ denotes the usual surface measure on the sphere $\mathbb{S}^{d}$. We denote by $\kappa_d  = \sigma(\mathbb{S}^{d}) = 2 \pi^{(d+1)/2}/\Gamma((d+1)/2)$ the total surface area of $\mathbb{S}^{d}$. With a slight abuse of notation, we shall also write
\begin{equation}\label{Def_sigma_r}
\sigma(r) := \sigma\big(\mathcal{B}_r(\zeta)\big) = \kappa_{d-1} \int_0^r (\sin t)^{d-1}\,\d t.
\end{equation}
Throughout this section we assume, without loss of generality, that $f$ {\it is real-valued and nonnegative} (or $+\infty)$.

\subsection{Preliminaries} \label{Prelim_11:49} If $f \in L^1(\mathbb{S}^{d})$, by Lebesgue differentiation we may modify it in a set of measure zero so that
\begin{equation}\label{Leb_Diff}
f(\xi) = \limsup_{\{r \to 0^+ \ :  \ \xi \in \overline{\mathcal{B}_r(\zeta)}\}} \ \intav{\mathcal{B}_r(\zeta)} f(\eta)\,\d \sigma(\eta)
\end{equation}
holds everywhere. Let us assume that is the case. For $f \in L^1(\mathbb{S}^{d})$ and $\xi \in \mathbb{S}^{d}$ let us define the set ${\bf B}_\xi$ as the set of closed balls that realize the supremum in the definition of the maximal function, that is 
\begin{equation}\label{20200619_10:30}
{\bf B}_\xi = \left\{\overline{\mathcal{B}_r(\zeta)}; \ \zeta \in  \mathbb{S}^{d}\, ,\,  r \geq 0\,,\,  \xi \in \overline{\mathcal{B}_r(\zeta)} \ : \   \widetilde{\mathcal{M}}f(\xi) = \, \intav{\, \overline{\mathcal{B}_r(\zeta)}} \, f(\eta)\,\d \sigma(\eta)\right\}.
\end{equation}
Here we consider the slight abuse of notation 
\begin{equation}\label{20200619_10:31}
\overline{\mathcal{B}_0(\xi)} := \{\xi\} \ \ {\rm and} \ \  \intav{\{\xi\}}  f(\eta)\,\d \sigma(\eta) := f(\xi),
\end{equation} 
in order to include the closed ball of radius zero as a potential candidate in the definition of ${\bf B}_\xi$. In light of \eqref{Leb_Diff} we always have that ${\bf B}_\xi$ is non-empty. Our first lemma holds for general Sobolev functions in $W^{1,1}(\mathbb{S}^{d})$ (not necessarily polar functions). 

\begin{lemma}\label{Lem4 - inner integral sphere}
Let $f \in W^{1,1}(\mathbb{S}^{d})$ be a nonnegative function that verifies \eqref{Leb_Diff} and let $\xi \in \mathbb{S}^{d}$ be a point such that $\widetilde{\mathcal{M}}f(\xi) > f(\xi)$. Assume that $\widetilde{\mathcal{M}}f$ is differentiable at $\xi$ and that $\overline{\mathcal{B}} \in {\bf B}_{\xi}$. Then 
$$\nabla \widetilde{\mathcal{M}}f(\xi)v = \, \intav{\mathcal{B}} \nabla f(\eta)\big(\!-(\eta\cdot v)\xi + (\eta\cdot \xi)v\big) \,\d \sigma(\eta)$$
for every $v \in \R^{d+1}$ with $v \perp \xi$. In particular,
$$\big|\nabla \widetilde{\mathcal{M}}f(\xi)\big| \leq \intav{\mathcal{B}} |\nabla f(\eta)| \,\d \sigma(\eta).$$
\end{lemma}

\begin{proof}
Observe first that the condition $\widetilde{\mathcal{M}}f(\xi) > f(\xi)$ implies that the ball $\mathcal{B}$ has positive radius. Without loss of generality let us assume that $|v| =1$. Let $R_t = R_{t,\xi, v}$ be the rotation of angle $t$ over the plane spanned by $\xi$ and $v$ that leaves the orthogonal complement invariant, i.e. 
$$R_t(\eta) = \big((\cos t) (\eta \cdot \xi) - (\sin t) (\eta \cdot v)\big)\xi + \big((\sin t) (\eta \cdot \xi) + (\cos t) (\eta \cdot v)\big)v + z(\eta)\,,$$
where $z(\eta)$ is the component of the vector $\eta$ that is orthogonal to the plane generated by $\xi$ and $v$. Then
\begin{align}\label{05302019_23:03}
\begin{split}
\nabla \widetilde{\mathcal{M}}f(\xi)v & = \lim_{t\to 0+}\frac{ \widetilde{\mathcal{M}}f(R_t\xi) - \widetilde{\mathcal{M}}f(\xi)}{t}\\
& \geq \lim_{t\to 0+} \frac{1}{t}\left( \intav{R_t(\mathcal{B})}f - \intav{\mathcal{B}}f\right)  \\
& =  \lim_{t\to 0+} \, \intav{\mathcal{B}} \frac{f(R_t\eta) - f(\eta) }{t} \, \d \sigma(\eta)\\
& = \, \intav{\mathcal{B}} \nabla f(\eta)\big(\!-(\eta\cdot v)\xi + (\eta\cdot \xi)v\big) \,\d \sigma(\eta).
\end{split}
\end{align}
The reverse inequality is obtained similarly by considering the limit as $t \to 0^-$. 
\end{proof}

\noindent {\it Remark}: The passage to the limit in \eqref{05302019_23:03} uses the fact that the difference quotients are bounded in $L^1$ by a multiple of the $L^1$-norm of the gradient of $f$, uniformly in $t$. With such a uniform bound one can establish the required limit by approximating $f$ by smooth $g$.

\subsection{Lipschitz case} Throughout this subsection we assume that our polar $f \in W^{1,1}(\mathbb{S}^d)$ is a Lipschitz function. Recalling that ${\bf e} = (1,0,0,\ldots, 0) \in \R^{d+1}$, for $\xi \in \mathbb{S}^d$ we write
$$\cos  \theta = \xi \cdot {\bf e}$$
with $\theta \in [0,\pi]$. Note that $\theta = \theta(\xi) = d({\bf e}, \xi)$ is the polar angle. We generally write $f(\xi)$ for the function on $\mathbb{S}^d$, and $f(\theta)$ for its polar version on $(0,\pi)$. We then have
$$|\nabla f(\xi)| =|f'(\theta)|$$
for a.e. $\xi \in \mathbb{S}^d \setminus \{{\bf e}, {\bf -e}\}$, and
$$\|\nabla f\|_{L^1(\mathbb{S}^d)} = \kappa_{d-1} \int_0^\pi |f'(\theta)| \,(\sin \theta)^{d-1}\,\d \theta.$$

\subsubsection{Estimates for small radii} For $\zeta \in \mathbb{S}^d$ let us define 
$$w(\zeta) = \min\big\{\theta(\zeta) \, ,\, \pi - \theta(\zeta)\big\} = \min \{d({\bf e}, \zeta), d({\bf -e}, \zeta)\}.$$
Let us define the auxiliary maximal operator $\widetilde{\mathcal{M}}^I $ by (recall convention \eqref{20200619_10:31})
\begin{equation}\label{20200619_18:46}
\widetilde{\mathcal{M}}^I f(\xi) = \sup_{\{\xi \in \overline{\mathcal{B}_r(\zeta)} \ : \ 0 \leq r \leq w(\zeta) /4\}}  \ \intav{\ \overline{\mathcal{B}_r(\zeta)}} \, f(\eta)\,\d \sigma(\eta).
\end{equation}
In analogy to \eqref{20200619_10:30}, for each $\xi \in \SS^{d}$ we define the set of good balls
\begin{equation*}
{\bf B}^I_\xi = \left\{\overline{\mathcal{B}_r(\zeta)}; \ \zeta \in  \mathbb{S}^{d}\, ,\,  0 \leq r \leq \frac{w(\zeta)}{4}\,;\,  \xi \in \overline{\mathcal{B}_r(\zeta)} \ : \   \widetilde{\mathcal{M}}^If(\xi) = \, \intav{\, \overline{\mathcal{B}_r(\zeta)}} \, f(\eta)\,\d \sigma(\eta)\right\}.
\end{equation*}
Notice that $\widetilde{\mathcal{M}}^I f$ is also a polar function. We consider the detachment set
$$\mathcal{E}_d := \big\{ \xi \in \mathbb{S}^d \setminus \{{\bf e}, {\bf -e}\} \ : \ \widetilde{\mathcal{M}}^I f(\xi) > f(\xi)\big\},$$
and its polar version, denoted by
$$\mathcal{E}_1 =\{\theta(\xi) = d({\bf e}, \xi) \ : \ \xi \in  \mathcal{E}_d\}.$$
One can check that $\widetilde{\mathcal{M}}^I f$ is a continuous function in $\mathbb{S}^d$.  Further qualitative properties of $\widetilde{\mathcal{M}}^I f$ are described in the next two results.
\begin{lemma}
$\widetilde{\mathcal{M}}^I f$ does not have a strict local maximum in $\mathcal{E}_1$.
\end{lemma}
\begin{proof}
The proof is identical to \cite[Lemma 3.10]{Lu2}.
\end{proof}

\begin{lemma}\label{loc_lipschitz_sphere}
$\widetilde{\mathcal{M}}^I f$ is locally Lipschitz in $\mathcal{E}_d$. 
\end{lemma}
\begin{proof}
Let $\xi \in \mathcal{E}_d$. Let $\overline{\mathcal{B}_r(\zeta)} \in {\bf B}^I_{\xi}$ with $r$ minimal. Then $r >0$ and it is possible to find a neighborhood $V$ of $\xi$ of the form $V = \{ \eta \in \mathbb{S}^d \ : \ \theta(\xi) - \varepsilon < \theta(\eta) < \theta(\xi) + \varepsilon\}$ such that: (i) $\varepsilon < r/100$ and (ii) if $\eta \in V$ and $\overline{\mathcal{B}_s(\omega)} \in {\bf B}^I_{\eta}$ then $s >99r/100$.

\smallskip

Let $\eta_1, \omega_2 \in V$. Let $S$ be the half great circle connecting ${\bf e}, \eta_1, -{\bf e}$. If $\eta_2 \in S$ is such that $d({\bf e}, \eta_2) = d({\bf e}, \omega_2)$ then we have $d(\eta_1, \eta_2) \leq d(\eta_1, \omega_2)$. Since $\widetilde{\mathcal{M}}^I f(\eta_2) = \widetilde{\mathcal{M}}^I f(\omega_2)$, for the purposes of proving Lipschitz continuity it suffices to work with $\eta_1, \eta_2 \in S$. Assume without loss of generality that $\widetilde{\mathcal{M}}^I f(\eta_1) >  \widetilde{\mathcal{M}}^I f(\eta_2)$. Let $\overline{\mathcal{B}_{r_1}(\zeta_1)} \in {\bf B}^I_{\eta_1}$ with $\zeta_1 \in S$. Then $\eta_2 \notin \overline{\mathcal{B}_{r_1}(\zeta_1)}$, and hence $\eta_2$ is not between $\zeta_1$ and $\eta_1$. It is also easy to see that we cannot have $\zeta_1$ between $\eta_1$ and $\eta_2$ due to conditions (i) and (ii) above. Hence we must have $\eta_1$ between $\zeta_1$ and $\eta_2$. We now choose a ball $\mathcal{B}_{r_2}(\zeta_2)$, with $\zeta_2 \in S$ lying between $\zeta_1$ and $\eta_2$, such that $\eta_2 \in \partial \overline{\mathcal{B}_{r_2}(\zeta_2)}$ and 
\begin{equation}\label{Eq_lem_loc_Lip_0}
r_2 = d(\zeta_2, \eta_2) = \min\left\{r_1, \frac{w(\zeta_2)}{4}\right\} 
\end{equation}
(one may think of moving the center $\zeta_1$ along $S$ in the direction of $\eta_2$ until finding the unique choice of $\zeta_2$). Note that $\zeta_2$ is in fact between $\zeta_1$ and $\eta_1$ and hence
\begin{align}\label{Eq_lem_loc_Lip_1}
r_2 = d(\zeta_2, \eta_2) = d(\zeta_1, \eta_1) - d(\zeta_1,\zeta_2) + d(\eta_1, \eta_2) \leq r_1 - d(\zeta_1,\zeta_2) + d(\eta_1, \eta_2).
\end{align}
If $r_2 = r_1$ in \eqref{Eq_lem_loc_Lip_0} then we have $ d(\zeta_1,\zeta_2) \leq d(\eta_1, \eta_2)$. In the other case we have
\begin{equation*}
r_2 = \frac{w(\zeta_2)}{4} \geq \frac{w(\zeta_1)}{4} - \frac{d(\zeta_1, \zeta_2)}{4} \geq r_1 - \frac{d(\zeta_1, \zeta_2)}{4},
\end{equation*}
and combining with \eqref{Eq_lem_loc_Lip_1} we obtain $d(\zeta_1, \zeta_2) \leq \frac{4}{3} d(\eta_1, \eta_2)$, which yields  $r_1 - r_2 \leq \frac{1}{3} d(\eta_1, \eta_2)$. We conclude by observing that 
\begin{align*}
\widetilde{\mathcal{M}}^I f(\eta_1) -  \widetilde{\mathcal{M}}^I f(\eta_2) & \leq \, \intav{\mathcal{B}_{r_1}(\zeta_1)} f - \intav{\mathcal{B}_{r_2}(\zeta_2)} f \\
& \leq \,\left( \ \intav{\mathcal{B}_{r_1}(\zeta_1)} f - \intav{\mathcal{B}_{r_2}(\zeta_1)} f\right) + \left(\ \intav{\mathcal{B}_{r_2}(\zeta_1)} f - \intav{\mathcal{B}_{r_2}(\zeta_2)} f\right)\\
& \lesssim_{d,r,f} \, d(\eta_1, \eta_2).
\end{align*}
\end{proof}

An adaptation of the argument in \cite[Section 5.4]{CS} then shows that $\widetilde{\mathcal{M}}^I f(\theta)$ is weakly differentiable in $(0,\pi)$, with weak derivative given by $\chi_{\mathcal{E}_1^c} f'(\theta) + \chi_{\mathcal{E}_1} \big(\widetilde{\mathcal{M}}^I f\big)'(\theta)$. In fact, if $\theta \in \mathcal{E}_1^c$ is a point of differentiability of $f$ (which are almost all points of $\mathcal{E}_1^c$) one can plainly see that $f'(\theta) = 0$, otherwise one could do better than $f(\theta)$ in the maximal function \eqref{20200619_18:46} and $\theta$ would belong to $\mathcal{E}_1$ instead. The weak derivative of $\widetilde{\mathcal{M}}^I f(\theta)$ is then simply $\chi_{\mathcal{E}_1} \big(\widetilde{\mathcal{M}}^I f\big)'(\theta)$. From Lemma \ref{Lem_radialization_sphere} below we have that $\widetilde{\mathcal{M}}^I f(\xi)$ is weakly differentiable in $\SS^d$. The next proposition establishes the desired control of the variation.
\begin{proposition}\label{Prop_M_aux}
The following inequality holds
$$\big\|\nabla \widetilde{\mathcal{M}}^I f\big\|_{L^1(\mathbb{S}^d)} \lesssim_d \|\nabla f\|_{L^1(\mathbb{S}^d)}.$$
\end{proposition}
\begin{proof}
The proof follows the outline of \cite[Lemma 3.5]{Lu2} with minor changes. We need to prove that 
$$\int_{\mathcal{E}_1} \big|\big(\widetilde{\mathcal{M}}^I f)'(\theta)\big| \,(\sin \theta)^{d-1}\,\d \theta \lesssim_d \int_{0}^{\pi} \big|f'(\theta)\big| \,(\sin \theta) ^{d-1}\,\d \theta.$$
We shall prove that 
\begin{equation}\label{Small_radii_eq1}
\int_{\mathcal{E}_1\cap [0,\pi/2]} \big|\big(\widetilde{\mathcal{M}}^I f)'(\theta)\big| \,(\sin \theta)^{d-1}\,\d \theta \lesssim_d \int_{0}^{\pi} \big|f'(\theta)\big| \,(\sin \theta)^{d-1}\,\d \theta
\end{equation}
and the proposition follows by symmetry. For $k \geq 1$, we define $\mathcal{E}_1^k = \mathcal{E}_1 \cap \left[\frac{\pi}{2^{k+1}}, \frac{\pi}{2^k}\right]$, and since $\mathcal{E}_1$ is open we may write ${\rm int}\big(\mathcal{E}_1^k\big) = \bigcup_{i=1}^{\infty} (a_i^k,b_i^k)$. We observe that $ \frac{ \,(\sin 2\theta)^{d-1}}{\,(\sin \theta)^{d-1}} \simeq_d 1$ for $\theta \leq \frac{\pi}{4}$. When $a_i^k = \frac{\pi}{2^{k+1}}$ or $b_i^k = \frac{\pi}{2^{k}}$ we observe, from the definition of the auxiliary operator in \eqref{20200619_18:46}, that
$$\widetilde{\mathcal{M}}^I f(\pi/2^{k+1})\ , \ \widetilde{\mathcal{M}}^I f(\pi/2^{k}) \leq \ \sup_{\theta(\xi)\, \in\, [\pi/2^{k+2}, \pi/2^{k-1}] } f(\xi)$$
for $k\geq 2$. These are the ingredients needed to run the argument in \cite[Lemma 3.5]{Lu2} in order to get
\begin{align}\label{Small_radii_eq2}
\int_{\mathcal{E}_1^k} \big|\big(\widetilde{\mathcal{M}}^I f)'(\theta)\big| \,(\sin \theta)^{d-1}\,\d \theta  \lesssim_d \int_{\pi/2^{k+2}}^{\pi/2^{k-1}} \big|f'(\theta)\big| (\sin \theta)^{d-1} \,\d \theta
\end{align}
for $k\geq 2$. In the case $k=1$ we must be a bit more careful when $b_i^1 = \pi/2$ by using the bound
$$ \widetilde{\mathcal{M}}^I f(\pi/2) \leq \ \sup_{\theta(\xi)\, \in\, [\pi/4, 3\pi/4] } f(\xi)\,,$$
which then yields
\begin{align}\label{Small_radii_eq3}
\begin{split}
\int_{\mathcal{E}_1^1} \big|\big(\widetilde{\mathcal{M}}^I f)'(\theta)\big| \,(\sin \theta)^{d-1}\,\d \theta & \leq \int_{\mathcal{E}_1^1} \big|\big(\widetilde{\mathcal{M}}^I f)'(\theta)\big| \,\d \theta \lesssim \int_{\pi/8}^{3\pi/4} \big|f'(\theta)\big| \,\d \theta\\
&  \lesssim_d \int_{\pi/8}^{3\pi/4} \big|f'(\theta)\big| (\sin \theta)^{d-1} \,\d \theta.
\end{split}
\end{align}
Finally, we add up \eqref{Small_radii_eq2} and \eqref{Small_radii_eq3} to get \eqref{Small_radii_eq1}.
\end{proof}

\subsubsection{Estimates for large radii - preliminary lemmas} The other crucial ingredient in the proof of Luiro \cite[Lemma 2.2 (v)]{Lu2} is the bound
$$\big|\nabla {\widetilde M} f(x)\big| \leq \frac{1}{|x|}\,  \intav{B} |\nabla f (y)| \, |y|\, \d y,$$
where $\overline{B} \ni x$ is a ball in which the maximal function is realized. The main difficulty in the case of $\mathbb{S}^d$ is in establishing a bound that will serve a similar purpose. This is accomplished in Lemma \ref{crucial_lemma_large_radii} below but before we actually get there we need a few preliminary lemmas. Recall the definition of $\sigma(r)$ in \eqref{Def_sigma_r}, and observe that $\sigma'(r) = \kappa_{d-1}(\sin r)^{d-1}$ is equal to the $(d-1)$-dimensional area of $\partial \mathcal{B}_r(\zeta)$.
\begin{lemma}\label{Lem_10_prep}
Let $\xi \in \mathbb{S}^d \setminus \{{\bf e}, {\bf -e}\}$ and let $\overline{\mathcal{B}_r(\zeta)}  \in {\bf B}_{\xi}$, with $\zeta$ in the half great circle determined by ${\bf e}$, $\xi$ and $-{\bf e}$. Assume that $0 \leq \theta(\zeta) < \theta(\xi)$, that $\xi \in \partial \mathcal{B}_r(\zeta)$, that $\widetilde{\mathcal{M}}f(\xi) > f(\xi)$ and that $\widetilde{\mathcal{M}}f$ is differentiable at $\xi$. Then
\begin{align*}
\nabla \widetilde{\mathcal{M}}f(\xi)(v(\xi, {\bf e})) = \frac{\sigma'(r)}{\sigma(r)} \intav{\mathcal{B}_r(\zeta)} \nabla f (\eta) (v(\eta, \zeta) )\,\frac{\sigma(d(\zeta,\eta))}{\sigma'(d(\zeta,\eta))}  \,\d\sigma(\eta),
\end{align*} 
where 
\begin{equation*}
v(\eta, \zeta) = \frac{\zeta - (\eta \cdot \zeta)\eta}{ |\zeta- (\eta \cdot \zeta)\eta|}
\end{equation*}
is the unit vector, tangent to $\eta$, in the direction of the geodesic that goes from $\eta$ to $\zeta$.
\end{lemma}
\begin{proof}
Since $\widetilde{\mathcal{M}}f(\xi) > f(\xi)$ we have $r>0$. Let $S$ be the great circle determined by ${\bf e}$ and $\xi$. For small $h \in \R$ we consider a rotation $R_h$ of angle $h$ in this circle (in the direction from $\xi$ to ${\bf e}$) leaving the orthogonal complement in $\R^{d+1}$ invariant, and write $\zeta - h := R_h(\zeta)$. The idea is to look at the following quantity 
\begin{align}\label{Eq_1_Lem_two}
\lim_{h \to 0} \frac{\intav{\mathcal{B}_{r+h}(\zeta-h)} f - \intav{\mathcal{B}_{r}(\zeta)}f}{h} =   \lim_{h \to 0}\frac{\intav{\mathcal{B}_{r+h}(\zeta-h)} f - \intav{\mathcal{B}_{r}(\zeta-h)} f +  \intav{\mathcal{B}_{r}(\zeta-h)} f -  \intav{\mathcal{B}_{r}(\zeta)}f}{h} .
 \end{align}
In principle we do not know that the limit above exists. We shall prove that it in fact exists using the right-hand side of \eqref{Eq_1_Lem_two}. Once this is established, the left-hand side of \eqref{Eq_1_Lem_two} tells us that this limit must be zero, since the numerator is always nonpositive regardless of the sign of $h$.

From Lemma \ref{Lem4 - inner integral sphere} (in particular, see computation \eqref{05302019_23:03}) we note that 
 \begin{align}\label{Eq_2_Lem_two}
 \lim_{h \to 0} \frac{\intav{\mathcal{B}_{r}(\zeta-h)} f -  \intav{\mathcal{B}_{r}(\zeta)} f }{h}= \nabla \widetilde{\mathcal{M}}f(\xi)(v(\xi, {\bf e})).
 \end{align}
Note also that
 \begin{align}\label{Eq_3_Lem_two}
 \begin{split}
\frac{\intav{\mathcal{B}_{r+h}(\zeta-h)} f - \intav{\mathcal{B}_{r}(\zeta-h)} f }{h} & =  \frac{\frac{1}{\sigma(r+h)} - \frac{1}{\sigma(r)}}{h}\int_{\mathcal{B}_{r+h}(\zeta-h)} f \ +\  \frac{1}{\sigma(r)} \frac{\int_{\mathcal{B}_{r+h}(\zeta-h)} f  - \int_{\mathcal{B}_{r}(\zeta-h)} f}{h}\\
& \to -\frac{\sigma'(r)}{\sigma(r)^2} \int_{\mathcal{B}_{r}(\zeta)} f \ +\  \frac{1}{\sigma(r)} \int_{\partial \mathcal{B}_r(\zeta)} f
\end{split}
\end{align}
as $h \to 0$. Hence the limit in \eqref{Eq_1_Lem_two} exists and is zero. Now we consider momentarily $\zeta$ as the north pole in the computation below and proceed with the standard polar coordinates on the sphere. Writing $\eta = (\cos\theta, \omega \sin \theta)$, with $\omega \in \mathbb{S}^{d-1}$ we use integration by parts to get
\begin{align}
\begin{split}\label{Eq_4_Lem_two}
\!\!\!\!\int_{\mathcal{B}_r(\zeta)} & \nabla f (\eta) \,(-v(\eta, \zeta) )\,\frac{\sigma(d(\zeta,\eta))}{\sigma'(d(\zeta,\eta))} \,\d\sigma(\eta) = \int_{\mathbb{S}^{d-1}}\!\int_0^r \frac{\partial f}{\partial \theta} (\theta, \omega)\left( \int_0^{\theta} (\sin t)^{d-1} \d t\right) \d\theta\, \d \sigma_{d-1}(\omega)\\
& = \int_{\mathbb{S}^{d-1}} f(r, \omega) \left(\int_0^{r} (\sin t)^{d-1} \d t\right) \d \sigma_{d-1}(\omega) - \int_{\mathbb{S}^{d-1}} \int_0^r  f(\theta, \omega) (\sin \theta)^{d-1} \d\theta\, \d \sigma_{d-1}(\omega)\\
& = \frac{\sigma(r)}{\sigma'(r)} \int_{\partial \mathcal{B}_r(\zeta)} f  - \int_{\mathcal{B}_r(\zeta)} f.
\end{split}
\end{align}
The lemma then plainly follows from \eqref{Eq_1_Lem_two}, \eqref{Eq_2_Lem_two}, \eqref{Eq_3_Lem_two} and \eqref{Eq_4_Lem_two}.
\end{proof}

We now state a basic geometric lemma.

\begin{lemma}\label{Lem_12_prep}
Denote by $\triangle ABC$ a geodesic triangle with vertices $A, B, C$, opposite geodesic side lengths $a,b,c$, and (geodesic) angles $ \hat{A}, \hat{B}, \hat{C}$.
\begin{itemize}
\item[(i)] There exist universal constants $\gamma >1$ and $\rho >0$ such that for every $\triangle ABC \subset \overline{\mathcal{B}_{\rho}({\bf e})}$ we have
$$a \sin \hat{B} \leq \gamma \, b.$$
\item[(ii)] Under the same hypotheses, if $\hat{B} \leq \frac{\pi}{2}$ we have 
$$\big|c -a \,\cos \hat{B}\big|  \leq b.$$ 
\end{itemize}
\end{lemma}
\begin{proof} Part (i). By the triangle inequality we have $a \leq 2 \rho$. Then, for any $\gamma >1$ we can choose $\rho$ small so that $\sin \theta \leq \theta \leq \gamma \sin \theta$ for $0 \leq \theta \leq 2\rho$. Using the spherical law of sines we have
$$a \sin \hat{B} \leq \gamma \sin a \sin \hat{B} = \gamma \sin b \sin \hat{A} \leq \gamma \sin b  \leq \gamma b.$$

\smallskip

\noindent Part (ii). Assume that $\rho$ is small. We shall prove that $\cos (c -a \,\cos \hat{B} ) \geq \cos b$, which shall imply that $|c -a \,\cos \hat{B}| \leq b$. By the spherical law of cosines we have
$$\cos b = \cos c \cos a + \sin c \sin a \cos \hat{B}.$$
Note that
$$\cos (c -a \,\cos \hat{B} )= \cos c \,\cos (a \,\cos \hat{B}) + \sin c \,\sin (a \,\cos \hat{B}).$$
Since $0 \leq a \,\cos \hat{B} \leq a$ we have that $\cos (a \,\cos \hat{B})  \geq \cos a$. Also, by elementary calculus we have $\sin (a \,\cos \hat{B}) \geq \sin a \cos \hat{B}$, and the result plainly follows from these estimates.
\end{proof}

We conclude this part with another elementary fact. 

\begin{lemma}\label{Lem13_prep}
We have 
$$ u(t) := \frac{\int_0^t (\sin s)^{d-1}\, \d s}{t \, (\sin t)^{d-1}} =  \frac{\sigma(t)}{ t \,\sigma'(t)} \simeq_d 1$$
for $0\leq t\leq1/4$. Moreover, $u$ is a $C^{\infty}$-function in this range.
\end{lemma}
\begin{proof}
Note that 
\begin{align*}
 \frac{\int_0^t (\sin s)^{d-1}\, \d s}{t (\sin t)^{d-1}} & = \frac{1}{t} \int_0^t \left(\frac{\sin s}{\sin t}\right)^{d-1}\,\d s = \frac{\sin t}{t} \int_0^1 a^{d-1} \frac{1}{(1 - a^2 (\sin t)^2)^{1/2}}\,\d a\,,
\end{align*}
and both $t \mapsto \frac{\sin t}{t}$ and $t\mapsto \int_0^1 a^{d-1} \frac{1}{(1 - a^2 (\sin t)^2)^{1/2}}\,\d a$ are smooth functions bounded above and below in the proposed range.
\end{proof}

\subsubsection{Estimates for large radii - main lemma} We are now in position to prove the key result of this subsection.

\begin{lemma}\label{crucial_lemma_large_radii} Let $\xi \in \mathbb{S}^d \setminus \{{\bf e}, {\bf -e}\}$ and let $\overline{\mathcal{B}_r(\zeta)}  \in {\bf B}_{\xi}$, with $\zeta$ in the half great circle determined by ${\bf e}$, $\xi$ and $-{\bf e}$. Assume that $0 \leq \theta(\zeta) < \theta(\xi)$, that $\xi \in \partial \mathcal{B}_r(\zeta)$, that $\widetilde{\mathcal{M}}f(\xi) > f(\xi)$ and that $\widetilde{\mathcal{M}}f$ is differentiable at $\xi$. There is a universal constant $\rho >0 $ such that if $\mathcal{B} = \mathcal{B}_r(\zeta) \subset \overline{\mathcal{B}_{\rho}({\bf e})}$ then
\begin{align}\label{Est_crucial_lemma}
\big| \nabla \widetilde{\mathcal{M}}f(\xi)\big| \lesssim_d \frac{1}{\theta(\xi)} \intav{\mathcal{B} } |\nabla f (\eta)| \,\theta(\eta)\,\d\sigma(\eta) +  \frac{r \,\theta(\zeta)}{\theta(\xi)} \intav{\mathcal{B} } |\nabla f (\eta)|\,\d\sigma(\eta).
\end{align}
\end{lemma}
\begin{proof}
From Lemma \ref{Lem_10_prep} we have
\begin{align}\label{Rep_1}
\nabla \widetilde{\mathcal{M}}f(\xi)(-v(\xi, {\bf e})) = \frac{\sigma'(r)}{\sigma(r)} \intav{\mathcal{B}} \nabla f (\eta) (-v(\eta, \zeta) )\,\frac{\sigma(d(\zeta,\eta))}{\sigma'(d(\zeta,\eta))}  \,\d\sigma(\eta).
\end{align} 
In the case $\zeta = {\bf e}$, estimate \eqref{Est_crucial_lemma} follows directly from \eqref{Rep_1} and Lemma \ref{Lem13_prep}. From now on we assume that $\zeta \neq  {\bf e}$. From Lemma \ref{Lem4 - inner integral sphere} we also know that 
\begin{align}\label{Rep_2}
\nabla \widetilde{\mathcal{M}}f(\xi)(-v(\xi, {\bf e})) = \intav{\mathcal{B}} \nabla f(\eta) \,S(\eta) \,\d \sigma(\eta),
\end{align}
with $S(\eta) = (\eta\cdot v(\xi, {\bf e}))\xi - (\eta\cdot \xi)v(\xi, {\bf e}).$ The idea is to compare the identities \eqref{Rep_1} and \eqref{Rep_2} in order to bound $\big|\nabla \widetilde{\mathcal{M}}f(\xi)\big| = \big|\nabla \widetilde{\mathcal{M}}f(\xi)(-v(\xi, {\bf e}))\big|$. To do so, we write the right-hand side of \eqref{Rep_2} as a sum of three terms, one being comparable to $\big|\nabla \widetilde{\mathcal{M}}f(\xi)\big|$, the second one being small, and the third one being close to the right-hand side of \eqref{Rep_1} in a suitable sense. We start by writing 
$$1 = \frac{\theta(\xi) - \theta(\zeta)}{r} = \frac{d({\bf e},\xi) - d({\bf e},\zeta)}{r}.$$
Let us define $v_1(\eta) = S(\eta) / |S(\eta)|$. We then have
\begin{align}\label{rep_3}
\begin{split}
\intav{\mathcal{B}} \nabla f(\eta) \,S(\eta) \,\d \sigma(\eta) & = \intav{\mathcal{B}} \nabla f(\eta) \,|S(\eta)|  \left(\frac{\theta(\xi) - \theta(\zeta)}{r}\right)v_1(\eta)\, \d \sigma(\eta)\\
& = \intav{\mathcal{B}} \nabla f(\eta) \,|S(\eta)| \, \frac{\theta(\xi)}{r}\,v_1(\eta)\, \d \sigma(\eta) \\
&  \ \ \  \ - \intav{\mathcal{B}} \nabla f(\eta) \,\big(|S(\eta)| - 1\big) \frac{\theta(\zeta)}{r}\,v_1(\eta)\, \d \sigma(\eta) - \intav{\mathcal{B}} \nabla f(\eta)  \frac{\theta(\zeta)}{r}\,v_1(\eta)\, \d \sigma(\eta).
\end{split}
\end{align}

\noindent {\it Step 1}. Let us start by bounding the quantity
$$\frac{\sigma'(r)}{\sigma(r)} \intav{\mathcal{B}} \nabla f (\eta) (-v(\eta, \zeta) )\,\frac{\sigma(d(\zeta,\eta))}{\sigma'(d(\zeta,\eta))}  \,\d\sigma(\eta) +\, \intav{\mathcal{B}} \nabla f(\eta)  \frac{\theta(\zeta)}{r}\,v_1(\eta)\, \d \sigma(\eta).$$
This last expression is equal to (recall the definition of $u$ in Lemma \ref{Lem13_prep})
\begin{align}\label{05172019_14:24}
\frac{\sigma'(r)}{\sigma(r)} \intav{\mathcal{B}} \nabla f (\eta) \big[d(\zeta,\eta)\, u(d(\zeta,\eta))\, (-v(\eta, \zeta)) + d({\bf e},\zeta) \,u(r)\,v_1(\eta) \big]  \,\d\sigma(\eta).
\end{align}

Note now that 
\begin{align*}
d(\zeta,\eta)\, u(d(\zeta,\eta))\, (-v(\eta, \zeta)) + d({\bf e},\zeta) \,u(r)\,v_1(\eta) & = u(d(\zeta,\eta)) \big[ d(\zeta,\eta)(-v(\eta, \zeta))  + d({\bf e},\zeta) v_1(\eta)\big] \\
& \ \ \ \  - d({\bf e},\zeta)  \big[ u(d(\zeta,\eta)) - u(r)\big] v_1(\eta).
\end{align*}

From Lemma \ref{Lem13_prep} we know that $u(t)$ is Lipschitz for $0\leq t \leq 1/4$. We then have $|u(d(\zeta,\eta)) - u(r)| \lesssim_d  r$ and another application of Lemma \ref{Lem13_prep} yields
\begin{align}\label{rep_5}
\frac{\sigma'(r)}{\sigma(r)} \left|\,\intav{\mathcal{B}} \nabla f (\eta)\, d({\bf e},\zeta)  \big[ u(d(\zeta,\eta)) - u(r)\big] v_1(\eta) \,\d\sigma(\eta)\right| \lesssim_d \intav{\mathcal{B}} |\nabla f (\eta)|\,d({\bf e},\zeta) \,\d\sigma(\eta).
\end{align}

Let us now deal with the remaining piece. Observe that 
\begin{align}
&d(\zeta,\eta)\,(-v(\eta, \zeta)) + d({\bf e},\zeta) \,v_1(\eta) = d(\zeta,\eta)\big( v_1(\eta)\cos \alpha + v_1(\eta)^*\sin \alpha\big) + d({\bf e},\zeta) \,v_1(\eta) \nonumber \\
& \ \ = \big[d(\zeta,\eta) v_1(\eta)\cos \beta + d({\bf e},\zeta) \,v_1(\eta)\big] + \big[d(\zeta,\eta)v_1(\eta)^*\sin \alpha\big]  + \big[d(\zeta,\eta)v_1(\eta)(\cos \alpha - \cos \beta)\big] \nonumber \\
& \ \ = [I] + [II] + [III], \label{three_terms}
\end{align}
where $\cos \alpha = -v(\eta, \zeta) \cdot v_1(\eta)$ ($0 \leq \alpha \leq \pi)$, $v_1(\eta)^*$ is unitary and orthogonal to $v_1(\eta)$ (in the plane determined by $v_1(\eta)$ and $v(\eta, \zeta)$), and $\cos \beta = v(\zeta, \eta) \cdot (-v(\zeta, {\bf e}))$ ($0 \leq \beta \leq \pi)$. Naturally, we may assume without loss of generality that $\eta \neq \zeta$. We now proceed with the analysis of the three terms in \eqref{three_terms}.

\smallskip

\noindent {\it Analysis of $[I]$.} Observe that 
$$|d(\zeta,\eta) v_1(\eta)\cos \beta + d({\bf e},\zeta) \,v_1(\eta)| = |d(\zeta,\eta)\cos \beta + d({\bf e},\zeta)|.$$
Consider the geodesic triangle with vertices ${\bf e}, \zeta, \eta$ (that has angle $\angle {\bf e} \zeta \eta = \pi - \beta$). Assuming $\rho$ small, if $\beta > \pi/2$ we may use Lemma \ref{Lem_12_prep} (ii) to find 
$$|d(\zeta,\eta)\cos \beta + d({\bf e},\zeta)| \leq d({\bf e},\eta).$$
In case $0 \leq \beta \leq \pi/2$ we have
\begin{align*}
0 \leq {\rm sgn}(\cos \beta) = {\rm sgn}\big[(\eta - (\zeta \cdot \eta)\zeta)\cdot (-{\bf e} + (\zeta \cdot {\bf e})\zeta)\big] = {\rm sgn} \big[\!-(\eta \cdot {\bf e}) + (\zeta \cdot {\bf e})(\zeta \cdot \eta)\big],
\end{align*}
which implies that 
\begin{align*}
\cos(\theta (\zeta)) =  (\zeta \cdot {\bf e}) \geq (\zeta \cdot {\bf e})(\zeta \cdot \eta) \geq (\eta \cdot {\bf e}) = \cos(\theta(\eta)).
\end{align*}
From this we conclude that $d({\bf e}, \zeta) = \theta (\zeta) \leq \theta (\eta) = d({\bf e}, \eta)$ and hence
\begin{align*}
|d(\zeta,\eta)\cos \beta + d({\bf e},\zeta)| \leq d(\zeta,\eta) + d({\bf e},\zeta) \leq (d({\bf e}, \zeta)  + d({\bf e}, \eta)) + d({\bf e},\zeta) \leq 3d({\bf e},\eta).
\end{align*}

\smallskip

\noindent {\it Analysis of $[II]$ and $[III]$.} We note that the angles $\alpha$ and $\beta$ are close, and it is important for our purposes to actually quantify this discrepancy. In order to do this, let us parametrize the points as follows. We write $\zeta = (\cos \theta, \sin \theta, {\bf 0})$, with ${\bf 0}  \in \R^{d-1}$, and $\eta = (\cos \theta_1, \sin \theta_1 \cos \varphi, \sin \theta_1 \sin \varphi\,\, \omega)$ with $\omega \in \mathbb{S}^{d-2} \subset \R^{d-1}$. Here we set $0 \leq \theta, \theta_1, \varphi \leq \pi$. Recall that in this notation we have ${\bf e} = (1,0, {\bf 0})$. We then have $-v(\zeta, {\bf e}) = (-\sin \theta, \cos \theta, {\bf 0})$. Recall also that the vector $v_1(\eta)$ is the unitary vector tangent to $\eta$ in the direction of the derivative of the curve that takes the point $\eta$ along the rotation in the first two coordinates (in the direction from ${\bf e}$ to $\zeta$). A direct computation yields
\begin{equation}\label{05172019_14:36}
S(\eta) = (-\sin \theta_1 \cos \varphi, \cos \theta_1, {\bf 0})
\end{equation}
and
$$v_1(\eta) = \frac{1}{\sqrt{1 - \sin^2 \theta_1\sin^2\varphi}} (-\sin \theta_1 \cos \varphi, \cos \theta_1, {\bf 0}).$$
Using that $v(\zeta, {\bf e}) \perp \zeta$ and $v_1(\eta) \perp \eta$ we then find
\begin{align*}
\cos \beta = v(\zeta, \eta) \cdot (-v(\zeta, {\bf e})) = \frac{\eta - (\eta \cdot \zeta)\zeta}{ |\eta- (\eta \cdot \zeta)\zeta|} \cdot (-v(\zeta, {\bf e})) = \frac{-\sin \theta \cos \theta_1+ \cos \theta \sin \theta_1\cos\varphi}{|\eta- (\eta \cdot \zeta)\zeta|}
\end{align*}
and
\begin{align*}
\cos \alpha = -v(\eta, \zeta) \cdot v_1(\eta) = \frac{-\zeta + (\eta \cdot \zeta)\eta}{ |-\zeta + (\eta \cdot \zeta)\eta|}\cdot v_1(\eta) = \frac{-\sin \theta \cos \theta_1+ \cos \theta \sin \theta_1\cos\varphi}{\sqrt{1 - \sin^2 \theta_1\sin^2\varphi} \,\, |\!-\zeta + (\eta \cdot \zeta)\eta|}.
\end{align*}
Since $|\eta- (\eta \cdot \zeta)\zeta| =  |-\zeta + (\eta \cdot \zeta)\eta| = \sqrt{1 - (\eta \cdot \zeta)^2}$, we plainly obtain that $|\cos \beta| \leq |\cos\alpha|$ and hence $\sin \alpha \leq \sin \beta$. Using Lemma \ref{Lem_12_prep} (i) we then find
\begin{align*}
|d(\zeta,\eta)v_1(\eta)^*\sin \alpha| \leq d(\zeta,\eta)\sin \beta \lesssim  d({\bf e}, \eta).
\end{align*}
This takes care of the term $[II]$ in \eqref{three_terms}. Finally, we recall that all the action takes place inside a small ball $\mathcal{B}_{\rho}({\bf e})$, which means that the angles $\theta$ and $\theta_1$ are small. This yields an estimate for the term $[III]$ of the form
\begin{align*}
|d(\zeta,\eta)& v_1(\eta)(\cos \alpha - \cos \beta)  |\lesssim  |\zeta - \eta||\cos \alpha - \cos \beta| \\
& = \frac{ \sqrt{2(1 - (\eta \cdot \zeta))}}{\sqrt{1 - (\eta \cdot \zeta)^2}} \ |\!-\sin \theta \cos \theta_1+ \cos \theta \sin \theta_1\cos\varphi|\left(\frac{1}{\sqrt{1 - \sin^2 \theta_1\sin^2\varphi}} - 1 \right)\\
& \lesssim \sin^2\theta_1\\
& \lesssim \theta_1 = d(\bf{e}, \eta).
\end{align*}
Combining \eqref{05172019_14:24}, \eqref{rep_5} and the bounds for the terms $[I], [II], [III]$ in \eqref{three_terms}, and using Lemma \ref{Lem13_prep}, we arrive at 
\begin{align}\label{05172019_14:51}
\begin{split}
& \left|\frac{\sigma'(r)}{\sigma(r)} \intav{\mathcal{B}} \nabla f (\eta) (-v(\eta, \zeta) )\,\frac{\sigma(d(\zeta,\eta))}{\sigma'(d(\zeta,\eta))}  \,\d\sigma(\eta) +\, \intav{\mathcal{B}} \nabla f(\eta)  \frac{\theta(\zeta)}{r}\,v_1(\eta)\, \d \sigma(\eta) \right| \\
&  \ \ \ \ \ \ \ \ \ \ \lesssim_d \intav{\mathcal{B}} |\nabla f (\eta)|\,\theta(\zeta) \,\d\sigma(\eta) + \frac{1}{r} \intav{\mathcal{B}} |\nabla f (\eta)|\,\theta(\eta) \,\d\sigma(\eta).
\end{split}
\end{align}

\noindent {\it Step 2}. We continue our analysis with the term
$$\intav{\mathcal{B}} \nabla f(\eta) \,\big(|S(\eta)| - 1\big) \frac{\theta(\zeta)}{r}\,v_1(\eta)\, \d \sigma(\eta).$$
From \eqref{05172019_14:36} we know that $|S(\eta)|^2 = \eta \cdot p(\eta)$, where $p(\eta)$ is the projection of $\eta$ over the plane generated by $\zeta$ and ${\bf e}$. Therefore
\begin{align}\label{05172019_14:48}
\begin{split}
& \left|\,\intav{\mathcal{B}} \nabla f(\eta) \,\big(|S(\eta)| - 1\big) \frac{\theta(\zeta)}{r}\,v_1(\eta)\, \d \sigma(\eta)\right|  \leq \intav{\mathcal{B}} |\nabla f(\eta)| \,\big(1 - |S(\eta)|^2\big) \frac{\theta(\zeta)}{r}\, \d \sigma(\eta) \\
&  \ \ \ \ \ \ \ \  \leq \intav{\mathcal{B}} \big|\nabla f(\eta)\big| \, \big|\eta \cdot (\eta - p(\eta))\big| \frac{\theta(\zeta)}{r}\, \d \sigma(\eta)  \\
& \ \ \ \ \ \ \ \  \leq \intav{\mathcal{B}} \big|\nabla f(\eta)\big| \, |\eta - p(\eta)| \frac{\theta(\zeta)}{r}\, \d \sigma(\eta)\\
& \ \ \ \ \ \ \ \ \leq \intav{\mathcal{B}} \big|\nabla f(\eta)\big| \,\theta(\zeta)\, \d \sigma(\eta). 
\end{split}
\end{align}

\noindent {\it Step 3}. Combining \eqref{Rep_1}, \eqref{Rep_2}, \eqref{rep_3}, \eqref{05172019_14:51} and \eqref{05172019_14:48} we find that 
\begin{align*}
\left|\,\intav{\mathcal{B}} \nabla f(\eta) \,|S(\eta)|  \frac{\theta(\xi)}{r}\,v_1(\eta)\, \d \sigma(\eta) \right| \ \lesssim_d \ \intav{\mathcal{B}} |\nabla f (\eta)|\,\theta(\zeta) \,\d\sigma(\eta) + \frac{1}{r} \intav{\mathcal{B}} |\nabla f (\eta)|\,\theta(\eta) \,\d\sigma(\eta)\,,
\end{align*}
and therefore
\begin{align*}
\left|\nabla \widetilde{\mathcal{M}}f(\xi)\right| & = \left|\,\intav{\,\mathcal{B}} \nabla f(\eta) \,S(\eta) \,\d \sigma(\eta) \right| \\
& \lesssim_d \frac{1}{\theta(\xi)} \intav{\mathcal{B} } |\nabla f (\eta)| \,\theta(\eta)\,\d\sigma(\eta) +  \frac{r \,\theta(\zeta)}{\theta(\xi)} \intav{\mathcal{B} } |\nabla f (\eta)|\,\d\sigma(\eta).
\end{align*}
This concludes the proof of the lemma.
\end{proof}

\subsubsection{Proof of Theorem \ref{Thm2-sphere} - Lipschitz case} We are now in position to move on to the proof of Theorem \ref{Thm2-sphere} when our initial datum $f$ is a Lipschitz function. In this case we also have $\widetilde{\mathcal{M}}f$ Lipschitz.
Consider the set $\mathcal{H}_d = \{\xi \in \mathbb{S}^{d} \ : \ \widetilde{\mathcal{M}}f(\xi) > \widetilde{\mathcal{M}}^If(\xi)\}$. In light of Proposition \ref{Prop_M_aux} it suffices to show that 
$$\int_{\mathcal{H}_d} \big| \nabla\widetilde{\mathcal{M}}f(\xi)\big| \,\d \sigma(\xi)  \lesssim_d \int_{\mathbb{S}^d} |\nabla f(\xi)|\,\d \sigma(\xi).$$
For each $\xi \in \mathbb{S}^{d} \setminus \{{\bf e}, {\bf -e}\}$ let us choose a ball $\overline{\mathcal{B}_{r_{\xi}}(\zeta_{\xi})} \in {\bf B}_{\xi}$ with $r_{\xi}$ minimal and, subject to this condition, with $\zeta_{\xi}$ in the half great circle connecting ${\bf e}, \xi, {\bf -e}$ in a way that $w(\zeta_{\xi}) = \min \{d({\bf e},\zeta_\xi), d({\bf -e},\zeta_\xi)\}$ is minimal. If there are two potential choices for $\zeta_{\xi}$ we choose the one with $0 \leq \theta(\zeta_{\xi}) \leq \theta(\xi)$.

\smallskip

First let us observe that we can restrict our attention to small balls. For $c>0$, define the set $\mathcal{R}_c = \{ \xi \in \mathcal{H}_d \setminus \{{\bf e}, {\bf -e}\}  \ : \  r_{\xi} \geq c \}$. By Lemma \ref{Lem4 - inner integral sphere} we find
\begin{align*}
\int_{\mathcal{R}_c} \big| \nabla \widetilde{\mathcal{M}}f(\xi)\big| \,\d \sigma(\xi) & \leq  \int_{\mathcal{R}_c} \,\frac{1}{\sigma(\mathcal{B}_{r_{\xi}}(\zeta_{\xi}))}\int_{\mathcal{B}_{r_{\xi}}(\zeta_{\xi})} |\nabla f(\eta)| \,\d \sigma(\eta)\, \d \sigma(\xi) \lesssim_{c,d} \int_{\mathbb{S}^d} |\nabla f(\eta)|\,\d \sigma(\eta).
\end{align*}
If $\xi \in \mathcal{H}_d \setminus \{{\bf e}, {\bf -e}\}$ and $r_{\xi}$ is small we must have $w(\zeta_{\xi}) < 4r_{\xi}$ (otherwise we would fall in the regime of the operator $\widetilde{\mathcal{M}}^I$). Assuming that $\xi \in \mathcal{H}_d \setminus \{{\bf e}, {\bf -e}\}$, that $\widetilde{\mathcal{M}}f$ is differentiable at $\xi$, and that $\nabla \widetilde{\mathcal{M}}f(\xi) \neq 0$ (which implies that $\xi \in \partial \mathcal{B}_{r_{\xi}}(\zeta_{\xi})$), we may restrict ourselves to the situation where $d({\bf e},\xi)\leq \rho$ or $d({\bf -e},\xi) \leq \rho$ (where $\rho$ is given by Lemma \ref{crucial_lemma_large_radii}). By symmetry let us assume that $\theta(\xi) = d({\bf e},\xi)\leq \rho$. We call such set $\mathcal{G}_d$ and further decompose it in $\mathcal{G}_d^- = \{\xi \in \mathcal{G}_d  \ : \ 0 \leq \theta(\zeta_{\xi}) < \theta(\xi)\}$ and $\mathcal{G}_d^+ = \{\xi \in \mathcal{G}_d  \ : \ 0 < \theta(\xi) < \theta(\zeta_{\xi})\}$. We bound the integrals over these two sets separately. 

\smallskip

\noindent {\it Step 1}. For $\mathcal{G}_d^+$ we use Lemma \ref{Lem4 - inner integral sphere} and proceed as follows:
\begin{align}\label{05172019_16:58}
\begin{split}
\int_{\mathcal{G}_d^+} \big| \nabla \widetilde{\mathcal{M}}f(\xi)\big| \,\d \sigma(\xi) & \leq \int_{\mathcal{G}_d^+} \, \intav{\mathcal{B}_{r_{\xi}}(\zeta_{\xi})} |\nabla f(\eta)| \,\d\sigma(\eta) \,\d \sigma(\xi)\\
& = \int_{\mathbb{S}^d} |\nabla f(\eta)| \int_{\mathcal{G}_d^+} \frac{\chi_{\mathcal{B}_{r_{\xi}}(\zeta_{\xi})}(\eta)}{\sigma(\mathcal{B}_{r_{\xi}}(\zeta_{\xi}))}\,\d\sigma(\xi)\,\d\sigma(\eta).
\end{split}
\end{align}
Note that $\theta(\eta) \geq \theta(\xi)$ in this case. Observe that 
\begin{equation}\label{05172019_16:54}
r_{\xi} > \frac{w(\zeta_{\xi})}{4} = \frac{\theta(\zeta_{\xi})}{4} \geq \frac{\theta(\xi)}{4},
\end{equation}
and also, by triangle inequality,
\begin{equation}\label{05172019_16:55}
r_{\xi} \geq \frac{d(\eta, \xi)}{2} \geq \frac{\theta(\eta)}{2} - \frac{\theta(\xi)}{2}.
\end{equation}
Dividing \eqref{05172019_16:55} by $2$ and adding up to \eqref{05172019_16:54} we get 
$$r_{\xi} \geq \frac{\theta(\eta)}{6}.$$
Returning to the computation \eqref{05172019_16:58} we have, for a fixed $\eta$, 
$$\int_{\mathcal{G}_d^+} \frac{\chi_{\mathcal{B}_{r_{\xi}}(\zeta_{\xi})}(\eta)}{\sigma(\mathcal{B}_{r_{\xi}}(\zeta_{\xi}))}\,\d\sigma(\xi) \leq \int_{\mathcal{B}_{\theta(\eta)}({\bf e})}\frac{1}{\sigma\big(\frac{\theta(\eta)} {6}\big)}\,\d\sigma(\xi) \simeq_d 1 ,$$
from which the required bound follows.

\smallskip

\noindent {\it Step 2}. We now bound the integral over $\mathcal{G}_d^-$ using Lemma 
\ref{crucial_lemma_large_radii}. If $\xi \in \mathcal{G}_d^-$ then
\begin{align}\label{05172019_17:30}
r_{\xi} \leq \theta(\xi) < 5 r_{\xi}.
\end{align}
We then have
\begin{align}\label{05172019_17:47}
\begin{split}
& \int_{\mathcal{G}_d^-} \big| \nabla \widetilde{\mathcal{M}}f(\xi)\big| \,\d \sigma(\xi) \lesssim_d \int_{\mathcal{G}_d^-}  \left(\frac{1}{\theta(\xi)} \intav{\mathcal{B}_{r_{\xi}}(\zeta_{\xi}) } |\nabla f (\eta)| \,\theta(\eta)\,\d\sigma(\eta) +  \frac{r_{\xi} \,\theta(\zeta_{\xi})}{\theta(\xi)} \intav{\mathcal{B}_{r_{\xi}}(\zeta_{\xi})} |\nabla f (\eta)|\,\d\sigma(\eta) \right) \d \sigma(\xi)\\
  &\ \  \lesssim \int_{\mathbb{S}^d}  |\nabla f (\eta)| \int_{\mathcal{G}_d^-} \frac{\chi_{\mathcal{B}_{r_{\xi}}(\zeta_{\xi})}(\eta) \, \theta(\eta)}{r_{\xi}\,\sigma(r_{\xi})}\,\d \sigma(\xi)\,\d \sigma(\eta) +  \int_{\mathbb{S}^d}  |\nabla f (\eta)| \int_{\mathcal{G}_d^-} \frac{\chi_{\mathcal{B}_{r_{\xi}}(\zeta_{\xi})}(\eta) \, \theta(\zeta_{\xi})}{\sigma(r_{\xi})}\,\d \sigma(\xi)\,\d \sigma(\eta).
 \end{split}
\end{align}
Using \eqref{05172019_17:30} and the fact that $\theta(\zeta_{\xi}) \leq \theta(\xi)$ in this case, we have, for a fixed $\eta$,
\begin{align}\label{05172019_17:48}
 \int_{\mathcal{G}_d^-} \frac{\chi_{\mathcal{B}_{r_{\xi}}(\zeta_{\xi})}(\eta) \, \theta(\zeta_{\xi})}{\sigma(r_{\xi})}\,\d \sigma(\xi) \leq  \int_{\mathcal{G}_d^-} \frac{\chi_{\mathcal{B}_{r_{\xi}}(\zeta_{\xi})}(\eta) \, \theta(\xi)}{\sigma(r_{\xi})}\,\d \sigma(\xi) \lesssim_d \int_0^{\rho} \frac{\theta \, (\sin \theta)^{d-1}}{\sigma (\theta)}\,\d \theta \lesssim_d 1
 \,,
\end{align}
where we used Lemma \ref{Lem13_prep} in the last inequality. For the other integral, we use \eqref{05172019_17:30}, the fact that $\theta(\eta) \leq \theta(\xi)$ in this case, and Lemma \ref{Lem13_prep} again to get
\begin{align}\label{05172019_17:49}
\int_{\mathcal{G}_d^-} \frac{\chi_{\mathcal{B}_{r_{\xi}}(\zeta_{\xi})}(\eta) \, \theta(\eta)}{r_{\xi}\,\sigma(r_{\xi})}\,\d \sigma(\xi) \leq \theta(\eta) \int_{\theta(\eta)}^{\rho} \frac{(\sin \theta)^{d-1}}{r_{\xi}\, \sigma(r_{\xi})} \, \d \theta \lesssim_d \theta(\eta) \int_{\theta(\eta)}^{\rho} \frac{1}{\theta^2}\,  \d \theta \lesssim 1.
\end{align}
Our desired inequality plainly follows from inserting the bounds given by \eqref{05172019_17:48} and \eqref{05172019_17:49} into \eqref{05172019_17:47}. This completes the proof of Theorem \ref{Thm2-sphere} in the Lipschitz case.

\subsection{Passage to the general case} 
We will be brief here since the outline is the same as in \S \ref{General case_Convolution}. The following lemma is the analogue of Lemma \ref{Lem_radialization} in the case of the sphere and we omit its proof.

\begin{lemma}\label{Lem_radialization_sphere}.
\begin{itemize}
\item[(i)] A polar function $f(\xi)$ is weakly differentiable in $\mathbb{S}^d \setminus \{{\bf e}, {\bf - e}\}$ if and only if its polar restriction $f(\theta)$ is weakly differentiable in $(0,\pi)$. In this case, the weak gradient $\nabla f$ of $f(\xi)$ and the weak derivative $f'$ of $f(\theta)$ are related by 
$$\nabla f(\xi) = f'(\theta(\xi)) (-v(\xi, {\bf e})) = f'(\theta(\xi)) \frac{-{\bf e} + (\xi \cdot {\bf e})\xi)}{|-{\bf e} + (\xi \cdot {\bf e})\xi)|}.$$

\item[(ii)] In the situation above, if $f(\xi)$ and $\nabla f(\xi)$ are locally integrable in neighborhoods of ${\bf e}$ and ${\bf - e}$, then $f$ is weakly differentiable in $\mathbb{S}^d$.
\end{itemize}
\end{lemma}

Consider now a (nonnegative) polar function $f(\xi)$ in $W^{1,1}(\mathbb{S}^d)$. Then, by Lemma \ref{Lem_radialization_sphere}, its polar version $f(\theta)$ is weakly differentiable in $(0,\pi)$ and verifies
\begin{align*}
\int_0^{\pi} |f'(\theta)|\, (\sin \theta)^{d-1}\,\d \theta < \infty.
\end{align*}
In particular, after a possible redefinition on a set of measure zero, one can take $f(\theta)$ continuous in $(0,\pi)$ (in fact, absolutely continuous in each compact interval of $(0,\pi)$). This is equivalent to assuming that $f(\xi)$ is continuous in $\mathbb{S}^d \setminus \{{\bf e}, {\bf - e}\}$.

\smallskip

In this case the detachment set 
$$\mathcal{D}_d := \{ \xi \in \mathbb{S}^d \setminus \{{\bf e}, {\bf -e}\} \ : \ \widetilde{\mathcal{M}} f(\xi) > f(\xi)\}$$
is an open set. One can also show that $\widetilde{\mathcal{M}}f$ is continuous in $\mathbb{S}^d \setminus \{{\bf e}, {\bf - e}\}$ (see the ideas in Step 3 of \S \ref{General case_Convolution}), being indeed locally Lipschitz in $\mathcal{D}_d$ (see the ideas in the proof of Lemma \ref{loc_lipschitz_sphere}, passage \eqref{05302019_23:03} and the remark thereafter). In particular, $\widetilde{\mathcal{M}}f$ is differentiable almost everywhere in $\mathcal{D}_d$. 

\smallskip

Let $\{f_n\} \subset C^{\infty}(\mathbb{S}^{d})$ be a sequence of nonnegative smooth functions such that $f_n \to f$ in $W^{1,1}(\mathbb{S}^d)$. We may simply assume that $f_n$ is given by the spherical convolution of $f$ with a smooth polar kernel $\varphi_n$ (say, non-increasing in the polar angle) of integral $1$ supported in the geodesic ball of radius $1/n$ centered at the north pole; see \cite[Chapter 2, Sections 2.1 and 2.3, and Proposition 2.6.4]{Dai_Xu} for details on the spherical convolution. We may also assume that $f_n \to f$ and $\nabla f_n \to \nabla f$ pointwise almost everywhere in $\mathbb{S}^{d}$ (say, outside a set $\mathcal{X} \subset \mathbb{S}^{d}$ of measure zero). Let $\xi \in \mathcal{D}_d \setminus \mathcal{X}$ be a point at which $\widetilde{\mathcal{M}}f$ is differentiable and all $\widetilde{\mathcal{M}}f_n$ are differentiable (this is still almost everywhere in $\mathcal{D}_d$). Note that for $n$ large we shall have $\xi \in \{\widetilde{\mathcal{M}} f_n(\xi) > f_n(\xi)\}$. We now observe that if $\mathcal{B}_{n} = \mathcal{B}_{r_n}(\zeta_n)$ is a ball that realizes the maximal function $\widetilde{\mathcal{M}} f_n(\xi)$ with $r_n \to r$ and $\zeta_n \to \zeta$, then we must have $r>0$ and the limiting ball $\mathcal{B}_{r}(\zeta)$ realizing the maximal function $\widetilde{\mathcal{M}} f(\xi)$. This plainly implies that 
$$ \widetilde{\mathcal{M}}f_n(\xi) \to \widetilde{\mathcal{M}}f(\xi)$$
as $n \to \infty$, and also, by Lemma \ref{Lem4 - inner integral sphere},
$$\nabla \widetilde{\mathcal{M}}f_n(\xi) \to \nabla \widetilde{\mathcal{M}}f(\xi)$$
as $n \to \infty$.

\smallskip

Since we have proved Theorem \ref{Thm2-sphere} for Lipschitz functions, using Fatou's lemma we have
\begin{align}\label{05202019_16:14}
\int_{\mathcal{D}_d} \big|\nabla \widetilde{\mathcal{M}}f(\xi)\big|\, \d \sigma(\xi) \leq \liminf_{n \to \infty}  \int_{\mathcal{D}_d} \big|\nabla \widetilde{\mathcal{M}}f_n(\xi)\big|\, \d \sigma(\xi) \lesssim_d \liminf_{n \to \infty} \| \nabla f_n\|_{L^1(\mathbb{S}^d)} = \| \nabla f\|_{L^1(\mathbb{S}^d)}.
\end{align}
This places us in position to adapt the one-dimensional argument of \cite[Section 5.4]{CS} to show that $\widetilde{\mathcal{M}}f(\theta)$ is weakly differentiable in $(0,\pi)$, with weak derivative given by
\begin{align}\label{05202019_16:15}
\chi_{\mathcal{D}_1^c} f'(\theta) + \chi_{\mathcal{D}_1} (\widetilde{\mathcal{M}}f)'(\theta)\,,
\end{align}
where $\mathcal{D}_1 = \{\theta(\xi) \ : \ \xi \in  \mathcal{D}_d\}$ is the polar version of $\mathcal{D}_d$. In fact, if $\theta \in \mathcal{D}_1^c$ is a point of differentiability of $f$ (which are almost all points of $\mathcal{D}_1^c$) one can verify that $f'(\theta) = 0$, otherwise $\theta$ would belong to $\mathcal{D}_1$ instead. The weak derivative of $\widetilde{\mathcal{M}}f(\theta)$ is then simply $\chi_{\mathcal{D}_1} \big(\widetilde{\mathcal{M}} f\big)'(\theta)$. This in turn implies that $\widetilde{\mathcal{M}}f$ is weakly differentiable in $\mathbb{S}^d \setminus \{{\bf e}, {\bf - e}\}$ by Lemma \ref{Lem_radialization_sphere}. From \eqref{05202019_16:14} and \eqref{05202019_16:15} we have 
\begin{align}\label{05202019_16:21}
\big\|\nabla \widetilde{\mathcal{M}}f\big\|_{L^1(\mathbb{S}^d)} \lesssim_d  \| \nabla f\|_{L^1(\mathbb{S}^d)},
\end{align}
which is our desired bound. From the Sobolev embedding we know that $f \in L^{d/(d-1)}(\mathbb{S}^d)$, and hence so does $\widetilde{\mathcal{M}}f$. In particular, $\widetilde{\mathcal{M}}f$ is locally integrable in $\mathbb{S}^d$. From \eqref{05202019_16:21} we already know that $\nabla \widetilde{\mathcal{M}}f$ is locally integrable in $\mathbb{S}^d$, and a further application of Lemma \ref{Lem_radialization_sphere} shows that $\widetilde{\mathcal{M}}f$ is in fact weakly differentiable in  $\mathbb{S}^d$, which completes our proof.

\section{Proof of Theorem \ref{Thm_conv_spheres}}\label{Section 4}
We now turn our attention to the proof of Theorem \ref{Thm_conv_spheres}. As presented in the introduction, the notation here is slightly different, as we denote our initial datum by $u_0$ and our maximal function by $u^*$. As usual, throughout this section, we assume that $u_0$ is real-valued and nonnegative (or $+\infty$).

\subsection{Lipschitz case} As in the proofs of the previous two theorems in this paper, we address first the case when our polar $u_0 \in W^{1,1}(\mathbb{S}^d)$ is a Lipschitz function. In this case we have that $u^*$ is a polar function that is also Lipschitz (see \cite[Lemma 16 (ii)]{CFS}). 

\subsubsection{A preliminary lemma} The following result will be important for our purposes.
\begin{lemma}\label{key_lemma_conv_sphere_heat}
Let $u_0: \mathbb{S}^d \to \R^+$ be a polar and Lipschitz function. Then, in polar coordinates,
$$\widetilde{\mathcal{M}}u_0\big(\tfrac{\pi}{2}\big) - u_0\big(\tfrac{\pi}{2}\big) \lesssim_d \|\nabla u_0\|_{L^1(\mathbb{S}^{d})}.$$
\end{lemma}
\begin{proof}
Let us assume that $\widetilde{\mathcal{M}}u_0\big(\tfrac{\pi}{2}\big) > u_0\big(\tfrac{\pi}{2}\big)$. First observe that
\begin{align*}
\widetilde{\mathcal{M}}u_0\big(\tfrac{\pi}{2}\big) - u_0\big(\tfrac{\pi}{2}\big) = \left(\widetilde{\mathcal{M}}u_0\big(\tfrac{\pi}{2}\big) - \sup_{\theta \in [\frac{\pi}{4}, \frac{3\pi}{4}]} u_0(\theta)\right) + \left( \sup_{\theta \in [\frac{\pi}{4}, \frac{3\pi}{4}]} u_0(\theta) - u_0\big(\tfrac{\pi}{2}\big)\right),
\end{align*}
and 
\begin{align*}
 \sup_{\theta \in [\frac{\pi}{4}, \frac{3\pi}{4}]} u_0(\theta) - u_0\big(\tfrac{\pi}{2}\big) &\leq \int_{\frac{\pi}{4}}^{\frac{3\pi}{4}} |u_0'(\theta)|\,\d\theta \lesssim_d \int_{\frac{\pi}{4}}^{\frac{3\pi}{4}} |u_0'(\theta)|\,(\sin \theta)^{d-1}\,\d\theta \lesssim_d \|\nabla u_0\|_{L^1(\mathbb{S}^{d})}.
\end{align*}
Therefore it suffices to bound $\widetilde{\mathcal{M}}u_0\big(\tfrac{\pi}{2}\big) - \sup_{\theta \in [\frac{\pi}{4}, \frac{3\pi}{4}]} u_0(\theta)$. Bringing things back to the notation of $\S$ \ref{Prelim_11:49}, let $\xi \in \mathbb{S}^d$ be such that $\theta(\xi) = \frac{\pi}{2}$ and let $\overline{\mathcal{B}} = \overline{\mathcal{B}_{r}(\zeta)} \in {\bf B}_{\xi}$. Let $\mathcal{Z} = \{\eta \in \mathbb{S}^d \ : \ \frac{\pi}{4} \leq \theta(\eta) \leq \frac{3\pi}{4}\}$. If $\mathcal{B} \subset \mathcal{Z}$, then $\widetilde{\mathcal{M}}u_0\big(\tfrac{\pi}{2}\big) - \sup_{\theta \in [\frac{\pi}{4}, \frac{3\pi}{4}]} u_0(\theta) \leq 0$ and we are done. Assume henceforth that $\mathcal{B} \not\subset \mathcal{Z}$ and that $\widetilde{\mathcal{M}}u_0\big(\tfrac{\pi}{2}\big) - \sup_{\theta \in [\frac{\pi}{4}, \frac{3\pi}{4}]} u_0(\theta) \geq 0$. Writing $\eta = (\cos \theta, (\sin \theta) \,\omega)$, with $\omega \in \mathbb{S}^{d-1}$, we define
$$\ell(\theta) = \int_{\mathbb{S}^{d-1}} \chi_{\mathcal{B}}(\eta)\, (\sin \theta)^{d-1}\,\d \sigma_{d-1}(\omega)$$
(that is, the $(d-1)$-dimensional measure of the intersection of $\mathcal{B}$ with the level set $d({\bf e}, \eta) = \theta$). We then have
\begin{align}\label{05212019_12:59}
\begin{split}
&\widetilde{\mathcal{M}}u_0\big(\tfrac{\pi}{2}\big)  = \intav{\mathcal{B}} u_0(\eta)\,\d\sigma(\eta) = \frac{1}{\sigma(\mathcal{B})} \int_0^{\pi} u_0(\theta) \, \ell(\theta) \,\d \theta\\
& \ \ \ \ \ = \frac{1}{\sigma(\mathcal{B})} \left( \int_0^{\frac{\pi}{4}} u_0(\theta) \, \ell(\theta) \,\d \theta + \int_{\frac{\pi}{4}}^{\frac{3\pi}{4}} u_0(\theta) \, \ell(\theta) \,\d \theta + \int_{\frac{3\pi}{4}}^{\pi} u_0(\theta) \, \ell(\theta) \,\d \theta \right)\\
& \ \ \ \ \ \leq \left(\sup_{\theta \in [\frac{\pi}{4}, \frac{3\pi}{4}]} u_0(\theta) \right)\frac{1}{\sigma(\mathcal{B})} \int_{\frac{\pi}{4}}^{\frac{3\pi}{4}} \ell(\theta) \,\d \theta + \frac{1}{\sigma(\mathcal{B})} \left( \int_0^{\frac{\pi}{4}} u_0(\theta) \, \ell(\theta) \,\d \theta + \int_{\frac{3\pi}{4}}^{\pi} u_0(\theta) \, \ell(\theta) \,\d \theta \right).
\end{split}
\end{align}
Now observe that 
\begin{align}\label{05212019_12:52}
\begin{split}
\int_{\frac{3\pi}{4}}^{\pi} u_0(\theta) \, \ell(\theta) \,\d \theta & = \int_{\frac{3\pi}{4}}^{\pi} \left(\int_{\frac{3\pi}{4}}^{\theta} u_0'(\tau) \,\d\tau +  u_0(\tfrac{3\pi}{4})\right)\, \ell(\theta) \,\d \theta\\
& = u_0(\tfrac{3\pi}{4}) \int_{\frac{3\pi}{4}}^{\pi}\ell(\theta) \,\d \theta  + \int_{\frac{3\pi}{4}}^{\pi}  u_0'(\tau)\left(\int_{\tau}^{\pi}\ell(\theta) \,\d \theta\right) \d \tau.
\end{split}
\end{align}
Plugging the bound
\begin{align*}
\int_{\tau}^{\pi}\ell(\theta) \,\d \theta \lesssim_d \int_{\tau}^{\pi}(\sin \theta)^{d-1} \,\d \theta \ \lesssim \ (\sin \tau)^{d-1}
\end{align*}
into \eqref{05212019_12:52} we get
\begin{align}\label{05212019_13:00}
\int_{\frac{3\pi}{4}}^{\pi} u_0(\theta) \, \ell(\theta) \,\d \theta \leq \left(\sup_{\theta \in [\frac{\pi}{4}, \frac{3\pi}{4}]} u_0(\theta)\right)\int_{\frac{3\pi}{4}}^{\pi}\ell(\theta) \,\d \theta + C_d\,\|\nabla u_0\|_{L^1(\mathbb{S}^{d})},
\end{align}
where $C_d$ is a universal constant. In an analogous way we obtain
\begin{align}\label{05212019_13:01}
\int_{0}^{\frac{\pi}{4}} u_0(\theta) \, \ell(\theta) \,\d \theta \leq \left(\sup_{\theta \in [\frac{\pi}{4}, \frac{3\pi}{4}]} u_0(\theta)\right)\int_{0}^{\frac{\pi}{4}}\ell(\theta) \,\d \theta + C_d\,\|\nabla u_0\|_{L^1(\mathbb{S}^{d})}.
\end{align}
Combining \eqref{05212019_12:59},  \eqref{05212019_13:00} and  \eqref{05212019_13:01} we get
\begin{align*}
\widetilde{\mathcal{M}}u_0\big(\tfrac{\pi}{2}\big) \leq \sup_{\theta \in [\frac{\pi}{4}, \frac{3\pi}{4}]} u_0(\theta) + C_d\, \|\nabla u_0\|_{L^1(\mathbb{S}^{d})},
\end{align*}
from where our result follows.
\end{proof}

\subsubsection{Proof of Theorem \ref{Thm_conv_spheres} - Lipschitz case} Assume $d\geq 2$ since the case $d=1$ has already been treated in \cite[Theorem 3]{CFS}. Define the detachment set (excluding the poles)
\begin{equation*}
\mathcal{A}_d = \{ \xi \in \mathbb{S}^d \setminus\{{\bf e}, {\bf -e}\} \ : \  u^*(\xi) > u_0(\xi)\}
\end{equation*}
and its one-dimensional polar version 
\begin{equation*}
\mathcal{A}_1 = \{ \theta(\xi)  \ : \ \xi \in \mathcal{A}_d \} \subset (0,\pi).
\end{equation*} 
These sets are open and from \cite[Lemma 17]{CFS} we know that $u^*$ is subharmonic on  $\mathcal{A}_d$. We write
\begin{equation*}
\mathcal{A}_1 = \bigcup_{i=0}^{\infty} (a_i, b_i)
\end{equation*}
as a countable union of disjoint open intervals. If $\frac{\pi}{2} \in \mathcal{A}_1$ we let $\frac{\pi}{2} \in (a_0,b_0)$ and let 
\begin{equation*}
\mathcal{A}_1^- = \bigcup_{(a_i, b_i) \subset \big(0, \tfrac{\pi}{2}\big)} (a_i, b_i)   \ \ \ {\rm and} \ \ \ \mathcal{A}_1^+ = \bigcup_{(a_i, b_i) \subset \big(\tfrac{\pi}{2}, \pi\big)} (a_i, b_i).
\end{equation*}
If $\frac{\pi}{2} \notin \mathcal{A}_1$ we just regard $(a_0, b_0)$ as empty, and keep $\mathcal{A}_1^{\pm}$ as above. 

\smallskip

Let $(a,b)$ denote a generic interval $(a_i,b_i)$ of this union. As in the proof of Theorem \ref{Thm1}, the subharmonicity implies that $u^*$ has no strict local maximum in $(a,b)$ and then there exists $\tau$ with $a \leq \tau \leq b$ such that $u^*$ is non-increasing in $[a, \tau]$ and non-decreasing in $[\tau, b]$. We then have $(u^*)'(\theta) \leq 0$ a.e. in $a < \theta < \tau$, and $(u^*)'(\theta) \geq 0$ a.e. in $\tau < \theta < b$. 
 
\smallskip

An important idea of this proof is to proceed via the comparison \eqref{20200625_11:06} to the uncentered Hardy-Littlewood maximal function when appropriate, and make use of the gradient bound established in Theorem \ref{Thm2-sphere}. We consider first the case when $(a,b) \subset \mathcal{A}_1^-$. Using integration by parts we get
\begin{align}
\int_a^b \big|(u^*)'(\theta)\big|\, (\sin \theta)^{d-1}&\,\d \theta  = -\int_a^{\tau} (u^*)'(\theta)\, (\sin \theta)^{d-1}\,\d \theta + \int_{\tau}^b (u^*)'(\theta)\, (\sin \theta)^{d-1}\,\d \theta \nonumber  \\
& = u^*(a)\, (\sin a)^{d-1} + u^*(b)\, (\sin b)^{d-1} - 2\,u^*(\tau)\, (\sin \tau)^{d-1} \nonumber \\
&  \ \ \ \ \ \ \ + \int_a^{\tau} u^*(\theta)\, \frac{\partial}{\partial \theta}  (\sin \theta)^{d-1}\, \d \theta - \int_\tau^{b} u^*(\theta)\, \frac{\partial}{\partial \theta}  (\sin \theta)^{d-1}\, \d \theta \nonumber \\
& \leq u_0(a)\, (\sin a)^{d-1} + u_0(b)\, (\sin b)^{d-1} - 2\,u_0(\tau)\, (\sin \tau)^{d-1} \label{05212019_16:33}\\
&  \ \ \ \ \ \ \ + \int_a^{\tau} \widetilde{\mathcal{M}}u_0(\theta)\, \frac{\partial}{\partial \theta}  (\sin \theta)^{d-1}\, \d \theta - \int_\tau^{b} u_0(\theta)\, \frac{\partial}{\partial \theta}  (\sin \theta)^{d-1}\, \d \theta \nonumber \\
& \leq \int_a^{b}\big|u_0'(\theta)\big|\, (\sin \theta)^{d-1}\,\d \theta + \int_a^{\tau} \big( \widetilde{\mathcal{M}}u_0(\theta) - u_0(\theta)\big)\, \frac{\partial}{\partial \theta}  (\sin \theta)^{d-1}\, \d \theta.\nonumber 
\end{align}
In the computation above we have taken advantage of the fact that $\frac{\partial}{\partial \theta}  (\sin \theta)^{d-1} \geq 0$. Note also that we have no problem if $a=0$ since $\lim_{a\to 0} u^*(a)\, (\sin a)^{d-1} = 0$ as $d \geq 2$. If we sum \eqref{05212019_16:33} over all the intervals $(a,b) \subset \mathcal{A}_1^-$ we find
\begin{align}\label{05212019_17:02}
& \!\!\int_{\mathcal{A}_1^-} \! \big|(u^*)'(\theta)\big|\, (\sin \theta)^{d-1}\,\d \theta  \leq \!\! \int_0^{\frac{\pi}{2}}\!\!\big|u_0'(\theta)\big|\, (\sin \theta)^{d-1}\,\d \theta + \int_0^{\frac{\pi}{2}} \!\!\big( \widetilde{\mathcal{M}}u_0(\theta) - u_0(\theta)\big)\, \frac{\partial}{\partial \theta}  (\sin \theta)^{d-1}\, \d \theta \nonumber\\
& = \int_0^{\frac{\pi}{2}}\big|u_0'(\theta)\big|\, (\sin \theta)^{d-1}\,\d \theta - \int_0^{\frac{\pi}{2}} \big( \big(\widetilde{\mathcal{M}}u_0\big)'(\theta) - u_0'(\theta)\big)\,  (\sin \theta)^{d-1}\, \d \theta + \Big(\widetilde{\mathcal{M}}u_0\big(\tfrac{\pi}{2}\big) - u_0\big(\tfrac{\pi}{2}\big)\Big)\\
& \lesssim_d \int_0^{\pi}\big|u_0'(\theta)\big|\, (\sin \theta)^{d-1}\,\d \theta, \nonumber
\end{align}
where we have used Theorem \ref{Thm2-sphere} and Lemma \ref{key_lemma_conv_sphere_heat}.

\smallskip

Finally we have to consider the case when $\frac{\pi}{2} \in \mathcal{A}_1$ and bound the integral $\int_{a_0}^{\pi/2} \big|(u^*)'(\theta)\big|\, (\sin \theta)^{d-1}\,\d \theta$. Let $\tau_0$ be the corresponding local minimum over the interval $(a_0, b_0)$. Let $c_0 = \min\{\tau_0, \frac{\pi}{2}\}$. Proceeding as in  \eqref{05212019_16:33} and \eqref{05212019_17:02} we obtain
\begin{align}\label{05212019_17:15}
\begin{split}
\!\!-\int_{a_0}^{c_0} \!\!&(u^*)'(\theta) (\sin \theta)^{d-1}\d \theta = u^*(a_0) (\sin a_0)^{d-1} \!\!- u^*(c_0) (\sin c_0)^{d-1} + \int_{a_0}^{c_0} \!u^*(\theta) \frac{\partial}{\partial \theta}  (\sin \theta)^{d-1} \d \theta\\
& \leq u_0(a_0)\, (\sin a_0)^{d-1} - u_0(c_0)\, (\sin c_0)^{d-1} + \int_{a_0}^{c_0} \widetilde{\mathcal{M}}u_0(\theta)\, \frac{\partial}{\partial \theta}  (\sin \theta)^{d-1}\, \d \theta\\
& = -\int_{a_0}^{c_0} u_0'(\theta)\, (\sin \theta)^{d-1}\,\d \theta + \int_{a_0}^{c_0} \big( \widetilde{\mathcal{M}}u_0(\theta) - u_0(\theta)\big)\, \frac{\partial}{\partial \theta}  (\sin \theta)^{d-1}\, \d \theta\\
& \lesssim_d \int_0^{\pi}\big|u_0'(\theta)\big|\, (\sin \theta)^{d-1}\,\d \theta.
\end{split}
\end{align}
The last estimate we need is the following
\begin{align}
\int_{c_0}^{\frac{\pi}{2}} (u^*)'(\theta)\, (\sin \theta)^{d-1}\,\d \theta & = u^*\big(\tfrac{\pi}{2}) - u^*(c_0)(\sin c_0)^{d-1} - \int_{c_0}^{\frac{\pi}{2}} u^*(\theta)\, \frac{\partial}{\partial \theta}  (\sin \theta)^{d-1}\, \d \theta \nonumber\\
& \leq \widetilde{\mathcal{M}}u_0\big(\tfrac{\pi}{2}\big) - u_0(c_0)(\sin c_0)^{d-1} - \int_{c_0}^{\frac{\pi}{2}} u_0(\theta)\, \frac{\partial}{\partial \theta}  (\sin \theta)^{d-1}\, \d \theta \nonumber \\
& = \Big(\widetilde{\mathcal{M}}u_0\big(\tfrac{\pi}{2}\big) - u_0\big(\tfrac{\pi}{2}\big)\Big) + \int_{c_0}^{\frac{\pi}{2}} u_0'(\theta)\,  (\sin \theta)^{d-1}\, \d \theta \label{05212019_17:17}\\
& \lesssim_d \int_0^{\pi}\big|u_0'(\theta)\big|\, (\sin \theta)^{d-1}\,\d \theta. \nonumber 
\end{align}

By combining \eqref{05212019_17:02}, \eqref{05212019_17:15} and \eqref{05212019_17:17}, and adding the integral over the set \{$u^* = u_0$\} we find
\begin{align*}
\int_{0}^{\frac{\pi}{2}} \big|(u^*)'(\theta)\big|\, (\sin \theta)^{d-1}\,\d \theta  \lesssim_d \int_0^{\pi}\big|u_0'(\theta)\big|\, (\sin \theta)^{d-1}\,\d \theta.
\end{align*}
By symmetry we then have
\begin{align*}
\int_{\frac{\pi}{2}}^{\pi} \big|(u^*)'(\theta)\big|\, (\sin \theta)^{d-1}\,\d \theta  \lesssim_d \int_0^{\pi}\big|u_0'(\theta)\big|\, (\sin \theta)^{d-1}\,\d \theta,
\end{align*}
and the proof is complete by adding these two estimates.

\subsection{Passage to the general case} The passage to the general case of a polar $f \in W^{1,1}(\mathbb{S}^d)$ follows closely the outline of \S \ref{General case_Convolution}, with Lemma \ref{Lem_radialization} replaced by Lemma \ref{Lem_radialization_sphere} when appropriate. We omit the details.

\section*{Acknowledgments}
E.C. acknowledges support from FAPERJ - Brazil. C.G.R. was supported by CAPES - Brazil. The authors are thankful to Juan Paucar for helpful comments.

\end{document}